\documentclass[3p]{elsarticle}
\usepackage[x11names,usenames,dvipsnames]{xcolor}
\usepackage{tikz}
\usepackage{framed}
\usepackage{url}
\usepackage{subfig}
\usepackage{hhline}
\usepackage[T1]{fontenc}
\usepackage{comment}
\usepackage{lineno, amsthm}
\usepackage{mathtools}
\usepackage{mathdots,enumerate}
\usepackage{textcomp}
\usepackage{siunitx}
\usepackage{amssymb,amsfonts}
\usepackage{amsmath,paralist,enumitem,mathrsfs,graphicx}
\usepackage{epstopdf}
\usepackage{pgfplots}
\newcommand\norm[1]{\left\lVert#1\right\rVert}

\usetikzlibrary{pgfplots.dateplot}
\renewcommand{\qedsymbol}{$\blacksquare$}
\newcommand\restr[2]{{
		\left.\kern-\nulldelimiterspace 
		#1 
		\vphantom{\big|} 
		\right|_{#2} 
}}
\makeatletter
\newcommand{\vas}{\bBigg@{3}}
\newcommand{\vast}{\bBigg@{4}}
\newcommand{\Vast}{\bBigg@{5}}
\makeatother
\newcommand{\upperRomannumeral}[1]{\uppercase\expandafter{\romannumeral#1}}
\newcommand{\lowerRomannumeral}[1]{\lowercase\expandafter{\romannumeral#1}}

\makeatletter
\let\save@mathaccent\mathaccent
\newcommand*\if@single[3]{%
	\setbox0\hbox{${\mathaccent"0362{#1}}^H$}%
	\setbox2\hbox{${\mathaccent"0362{\kern0pt#1}}^H$}%
	\ifdim\ht0=\ht2 #3\else #2\fi
}
\newcommand*\rel@kern[1]{\kern#1\dimexpr\macc@kerna}
\newcommand*\widebar[1]{\@ifnextchar^{{\wide@bar{#1}{0}}}{\wide@bar{#1}{1}}}
\newcommand*\wide@bar[2]{\if@single{#1}{\wide@bar@{#1}{#2}{1}}{\wide@bar@{#1}{#2}{2}}}
\newcommand*\wide@bar@[3]{%
	\begingroup
	\def\mathaccent##1##2{%
		\let\mathaccent\save@mathaccent
		\if#32 \let\macc@nucleus\first@char \fi
		\setbox\z@\hbox{$\macc@style{\macc@nucleus}_{}$}%
		\setbox\tw@\hbox{$\macc@style{\macc@nucleus}{}_{}$}%
		\dimen@\wd\tw@
		\advance\dimen@-\wd\z@
		\divide\dimen@ 3
		\@tempdima\wd\tw@
		\advance\@tempdima-\scriptspace
		\divide\@tempdima 10
		\advance\dimen@-\@tempdima
		\ifdim\dimen@>\z@ \dimen@0pt\fi
		\rel@kern{0.6}\kern-\dimen@
		\if#31
		\overline{\rel@kern{-0.6}\kern\dimen@\macc@nucleus\rel@kern{0.4}\kern\dimen@}%
		\advance\dimen@0.4\dimexpr\macc@kerna
		\let\final@kern#2%
		\ifdim\dimen@<\z@ \let\final@kern1\fi
		\if\final@kern1 \kern-\dimen@\fi
		\else
		\overline{\rel@kern{-0.6}\kern\dimen@#1}%
		\fi
	}%
	\macc@depth\@ne
	\let\math@bgroup\@empty \let\math@egroup\macc@set@skewchar
	\mathsurround\z@ \frozen@everymath{\mathgroup\macc@group\relax}%
	\macc@set@skewchar\relax
	\let\mathaccentV\macc@nested@a
	\if#31
	\macc@nested@a\relax111{#1}%
	\else
	\def\gobble@till@marker##1\endmarker{}%
	\futurelet\first@char\gobble@till@marker#1\endmarker
	\ifcat\noexpand\first@char A\else
	\def\first@char{}%
	\fi
	\macc@nested@a\relax111{\first@char}%
	\fi
	\endgroup
}

\usepackage{indentfirst}
\usepackage{hyperref}
\usepackage{epstopdf}

\theoremstyle{plain}
\newtheorem{theorem}{Theorem}

\newtheorem{lemma}[theorem]{Lemma}

\newtheorem{corollary}[theorem]{Corollary}

\newtheorem{rem}{Remark}
\theoremstyle{definition}

\modulolinenumbers[2]
\newenvironment{myproof}[2] {\paragraph{Proof of {#1} {#2}}}{\hfill\qedsymbol}
\makeatletter
\def\ps@pprintTitle{%
	\let\@oddhead\@empty
	\let\@evenhead\@empty
	\def\@oddfoot{\footnotesize\itshape
		\ifx\@empty\@empty
		\else\@journal\fi\hfill}%
	\let\@evenfoot\@oddfoot	
}
\makeatother
\title{Probability computation for high--dimensional semilinear SDEs driven by  isotropic $\alpha-$stable processes via mild Kolmogorov equations}

\author{Alessandro Bondi\corref{mycorrespondingauthor}}

\cortext[mycorrespondingauthor]{Classe di Scienze, Scuola Normale Superiore di Pisa, $56126$ Pisa, Italy. Email: {\tt alessandro.bondi@sns.it}} 
\begin{document}
\begin{abstract}
	Semilinear, $N-$dimensional stochastic differential equations (SDEs) driven by additive Lévy noise are investigated. Specifically, given $\alpha\in\left(\frac{1}{2},1\right),$ the interest is on SDEs driven by $2\alpha-$stable, rotation--invariant processes obtained by subordination of a Brownian motion. An original connection between the time--dependent Markov transition semigroup associated with their solutions and Kolmogorov backward equations in mild integral form is established via regularization--by--noise techniques. Such a link is the starting point for an iterative method which allows to approximate probabilities related to the SDEs with a single batch of Monte Carlo simulations as several parameters change, bringing a compelling computational advantage over the standard Monte Carlo approach. This method also pertains to the numerical computation of solutions to high--dimensional integro--differential Kolmogorov backward equations. The scheme, and in particular the first order approximation it provides, is then applied for two nonlinear vector fields and shown to offer  satisfactory results in dimension $N=100$.
	\\[1ex] 
	\\[1ex] 
\end{abstract}
	\maketitle

	\section{Introduction}
	In this paper, we are concerned with the study of quantities related to the $N-$dimensional, semilinear stochastic differential equation (SDE) 
	\begin{equation}\label{i1}
		\begin{cases}
			dX_t=\left(AX_t+B_0\left(t,X_t\right)\right)dt+\sqrt{Q}\,dW_{L_t},&t\in\left[s,T\right],\\
			X_s=x\in\mathbb{R}^N,
		\end{cases}
	\end{equation}
	with a specific interest in the case $N$ high. Here, given $\alpha\in\left(\frac{1}{2},1\right)$, $L$ is an $\alpha-$stable subordinator (i.e., an increasing Lévy process) independent from $\left(\beta^n\right)_{n=1,\dots, N}$, which in turn are independent Brownian motions; we write $W=\left[\begin{matrix}
		\beta^1,\dots,\beta^N
	\end{matrix}\right]^\top$. All these processes are defined in a common complete probability space $\left(\Omega,\mathcal{F},\mathbb{P}\right),$ which we endow with the minimal augmented filtration generated by the subordinated Brownian motion $W_L$.
	Moreover, $T>0$ is a finite time horizon and $s\in\left[0,T\right]$ is the initial time. 
	As for $A, Q \in\mathbb R^{N\times N}$, they  are diagonal matrices with $A$ negative--definite and $Q$ positive--definite. For our numerical experiments we will consider $Q=\sigma^2 \text{Id},$ being $\text{Id}\in\mathbb{R}^{N\times N}$ the identity matrix, so that $\sigma>0$ is a parameter describing the strength of the noise. Finally, the nonlinear bounded vector field $B_0\colon \left[0,T\right]\times \mathbb{R}^N\to \mathbb{R}^N$ is subject to suitable regularity conditions which will be specified in the sequel and guarantee, among other things, the existence of a pathwise unique solution of \eqref{i1}: it will be denoted by $X^{s,x}=\left(X^{s,x}_t\right)_{t\in\left[s,T\right]}$.
	
	Connected to the SDE in \eqref{i1}, we have the following \emph{Kolmogorov backward equation}:
	\begin{equation}\label{i2}
		\begin{cases}
			\partial_su\left(s,x\right)=-\left\langle Ax+B_0\left(s,x\right), \nabla^\top u \left(s,x\right)\right\rangle 
			\\
			\qquad\qquad\qquad\qquad - \int_{\mathbb{R}^N}\left[u\left(s,x+\sqrt{Q}z\right)-u\left(s,x\right)-1_D\left(z\right)\nabla u\left(s,x\right)\sqrt{Q}z\right]\nu\left(dz\right),&s\in\left[0,t\right),\\
			u\left(t,x\right)=\phi\left(x\right),& x\in\mathbb{R}^N,
		\end{cases}
	\end{equation}
	where  $\phi\colon\mathbb{R}^N\to \mathbb{R}$, $D=\left\{z\in\mathbb{R}^N,\, \left|z\right|\le 1 \right\}$ is the closed unit ball and we fix $t \in\left[0,T\right]$. Here $\nu\left(dz\right)$ is the Lévy measure of $W_L$, and up to a positive multiplicative constant is of the form $\nu\left(dz\right)= \left|z\right|^{-\left(N+2\alpha\right)}dz$ (see, e.g., \cite[Theorem 30.1]{Sato}). The link between the equations in \eqref{i1} and  \eqref{i2} is provided by Theorem \ref*{kfb} \emph{\ref{back_k} }below (see also the book \cite{Ku19} for related results), where it is shown that the time--dependent Markov transition semigroup  $\mathbb{E}\left[\phi\left(X_t^{s,x}\right)\right]$ associated with \eqref{i1} satisfies \eqref{i2} in the closed interval $\left[0,t\right]$ for every $\phi\in C^3_b\left(\mathbb{R}^N\right)$. Moreover,  we are able to extend the validity of this connection in $\left[0,t\right)$ to every function $\phi\in\mathcal{B}_b\left(\mathbb{R}^N\right)$ through an original procedure based on regularization--by--noise and a mild, integral formulation of \eqref{i2} (see Remark~\ref{rem1}). 
	
	In the present work, we are precisely interested in these expected values, with particular attention to the case $\phi\left(x\right)=1_{\left\{\left|x\right|>R\right\}}$ (for some threshold $R>0$), where one has  $\mathbb{E}\left[\phi\left(X^{s,x}_t\right)\right]=\mathbb{P}\left(\left|X^{s,x}_t\right|>R\right)$. Hence we want to describe a method which allows to compute probabilities related to the solution of the SDE \eqref{i1}. Trying to get an estimate of them by numerically solving the integro--differential equation \eqref{i2} is a typical example of \emph{curse of dimensionality} (CoD), and since we intend to deal with a high dimension (in the simulations we take $N=100$), this is an unfeasible way to proceed. The canonical approach to tackle our problem is the Monte Carlo method: several paths of $X^{s,x}$ are simulated by the Euler--Maruyama scheme with a fine time step, and then the final points of these trajectories are averaged to get an approximation of the desired expected values by virtue of the strong law of large numbers. However, if we were to follow this scheme (which is known to be free of the CoD), then we would have to start over the procedure every time we change the starting point $x$ and the starting time $s$, the noise strength $\sigma$ and even the nonlinearity $B_0$, a practice that is very common in a wide range of applications including weather forecasts and calibration of financial models (see \cite{BBGJJ} and references therein). In order to overcome this setback, we aim to extend to our framework the ideas developed in the papers \cite{LRF1, LRF2} for the Gaussian case, namely we search for an iterative scheme which relies on a single bulk of Monte Carlo simulations independent from the aforementioned parameters. Specifically, to approximate the value of the iterates $v^n_s\left(t,x\right),\,n\in\mathbb{N}\cup\left\{0\right\},$ we just need to simulate once and for all, using the Euler--Maruyama scheme, a large number of sample paths of the subordinator $L$ and of the stochastic convolution $\widetilde{Z}^{0}_t=\int_{0}^{t}e^{\left(t-r\right)A}dW_{L_r},\,t\in\left[0,T\right]$, which is the unique (up to indistinguishability) solution of the linear SDE
	\[
	d\widetilde{Z}^{0}_t=A\widetilde{Z}^{0}_tdt+dW_{L_t},\qquad \widetilde{Z}_0^{0}=0.
	\]
	
	The main novelty of the approach that we propose consists in the structure of the noise $W_L$, which is a $2\alpha-$stable, rotation--invariant Lévy process (cfr. \cite[Example $30.6$]{Sato}). In particular, the introduction of $L$ considerably complicates the framework compared to the Brownian one treated in \cite{LRF1, LRF2}. This fact leads us to develop an original procedure --essentially based on conditioning with respect to the $\sigma-$algebra generated by the subordinator-- to get an expression for the iterates which is suitable for applications. Moreover, the theoretical foundation of the iterative method analyzed in this work, Theorem \ref{connection}, has a remarkable interest on its own. Indeed, it establishes a connection between the time--dependent Markov transition semigroup associated with \eqref{i1} and a mild, integral formulation of \eqref{i2} (see Equation \eqref{Kolm_mild}) that, at the best of our knowledge, is new when it comes to isotropic Lévy processes.
	
	The paper is structured as follows. Section \ref{pre} describes the setting and recalls the main concepts that will be widely used in the rest of the paper. In addition, it introduces the integral formulation of the Kolmogorov equation \eqref{i2} and shows its well--posedness. Next, in Section \ref{a_conn} (see Theorem \ref{connection}) we provide the probabilistic interpretation of \eqref{i2} in mild form, along with other interesting regularization--by--noise results for SDEs driven by subordinated Wiener processes.  In Section \ref{ite_sc} we define the iterative scheme and prove its convergence to the expected values that we are trying to approximate. Next, Section \ref{first} is concerned with the computation of the first iterate $v^1_s\left(t,x\right)$; it is divided into two subsections referring to the deterministic and random time--shifts, respectively. Its results are used in Section \ref{sec_3} as the base case for the induction argument that allows to calculate $v^n_s\left(t,x\right)$ (see Theorem \ref{general}). The last part (Section \ref{simulations}) is devoted to numerical experiments in dimension $N=100$ for two  choices of the nonlinear vector field $B_0$, with particular attention on the improvements provided by the first iteration over the linear approximation corresponding to the Ornstein–Uhlenbeck (hereafter OU) processes.
	Finally, \ref{appendix} contains the proof of Lemma \ref{hoc}.
	
	\textbf{Notation:} Let $d,m,n\in\mathbb{N}$. In this paper, elements of $\mathbb{R}^d$ are columns vectors. For any $u,v\in\mathbb{R}^d$, we denote by $\left|u\right|$ the Euclidean norm and by $\left\langle u,v\right\rangle=u^\top v$ the standard scalar product. For a matrix $A\in\mathbb{R}^{d\times m},$ $\left|A\right|=\sup_{x\in\mathbb{R}^m \,:\, \left|x\right|=1}\left|Ax\right|$ is the operator norm. Given a vector field $B\colon \mathbb{R}^d\to\mathbb{R}^{m\times n}$, the uniform norm is $\norm{B}_\infty=\sup_{x\in\mathbb{R}^d}\left|B\left(x\right)\right|$. In particular, if $n=1$ then the Jacobian matrix is denoted by $DB\in\mathbb{R}^{m\times d}$, and $D_hB=DBh,\,h\in\mathbb{R}^d$; if also $m=1$ (so that $B$ is a scalar function) then the gradient $\nabla B$ is a row vector and $D^2B\in\mathbb{R}^{d\times d}$ represents the Hessian matrix. For an integer $k\in\mathbb{N}\cup \left\{0\right\}$, the space $C_b^k\left(\mathbb{R}^d;\mathbb{R}^{m\times n}\right)$ is constituted by the continuous vector fields $B$ which are bounded, continuously differentiable up to order $k$ with bounded derivatives. Taken $h=1,\dots,k$ and $B\in C_b^k\left(\mathbb{R}^d;\mathbb{R}^{m\times n}\right)$, we write $\norm{\partial^hB}_\infty=\sup_{i,j,\mathbf{h}}\norm{\partial_{\mathbf{h}}B_{i,j}}_\infty$, where  $B=\left(B_{i,j}\right),\,i=1,\dots,m,\,j=1,\dots,n$ and $\mathbf{h}\in\left(\mathbb{N}\cup \left\{0\right\}\right)^d$ is a multi--index with length $\norm{\mathbf{h}}_1=h$.
	\section{Preliminaries and  Kolmogorov backward equation in mild form} \label{pre}
	Fix $N\in\mathbb{N}$ and a complete probability space $\left(\Omega,\mathcal{F},\mathbb{P}\right)$. Consider  $N$ independent Brownian motions $\left(\beta^n\right)_{n=1,\dots,N}$: we write $W=\left[
		\beta^1,\dots,\beta^N
	\right]^\top$. Moreover, for $\alpha\in\left(0,1\right)$ we take a strictly $\alpha-$stable subordinator $L=\left(L_t\right)_{t\ge0}$ independent from $\left(\beta^n\right)_n$, and denote by $\mathcal{F}^L$ the augmented $\sigma-$algebra it generates, i.e., $\mathcal{F}^L=\sigma\left(\mathcal{F}^L_0\cup \mathcal{N}\right)$, where $\mathcal{F}^L_0$ is the natural $\sigma-$algebra generated by $L$ and $\mathcal{N}$ is the family of $\mathcal{F}-$negligible sets. In other words, $L$ is an increasing Lévy process with (cfr. \cite[Example $24.12$]{Sato})
\begin{equation}\label{subo}
	\mathbb{E}\left[e^{iuL_1}\right]=\exp\left\{-\bar{\gamma}^\alpha\left|u\right|^\alpha\left(1-i\tan\frac{\pi\alpha}{2}\text{sign}\,u\right)\right\},\quad u\in\mathbb{R}, \text{ for some }\bar{\gamma}>0.
\end{equation}
Let us introduce the diagonal matrices $A=-\text{diag}\left[
\lambda_1,\dots,\lambda_N\right]$ and $Q=\text{diag}\left[
\sigma^2_1,\dots,\sigma^2_N\right]$, with $0<\lambda_1\le\dots\le\lambda_N$ and 
$\sigma^2_n>0,\,n=1,\dots,N$. We endow $\Omega$ with the minimal augmented filtration $\mathbb{F}=\left(\mathcal{F}_t\right)_{t\ge0}$ generated by $W_L$, which means $\mathcal{F}_t=\sigma\left(\mathcal{F}_{0,t}^{W_L}\cup \mathcal{N}\right)$ for $t\ge0$, with $\left(\mathcal{F}^{W_L}_{0,t}\right)_{t\ge0}$ being the natural filtration of $W_L$.

	Given $T>0$ and a continuous function $f\colon\left[0,T\right]\to\mathbb{R}^N$, if $x\in\mathbb{R}^N$ and  $0\le s<T$ then $Z^{s,x}=\left(Z^{s,x}_t\right)_{t\in \left[s,T\right]}$ is the OU process starting from $x$ at time $s$, i.e., it is the unique solution of the next linear SDE
	\begin{equation}\label{OU}
		dZ_t^{s,x}=\left(AZ_t^{s,x}+f\left(t\right)\right)dt+\sqrt{Q}\,dW_{L_t},\quad Z^{s,x}_s=x.
	\end{equation}
	We denote by $R=\left(R_{s,t}\right),\,{0\le s \le t\le T},$ the time--dependent, Markov transition semigroup associated with this family of processes:
	\begin{equation}\label{mse}
		R_{s,t}\phi=\mathbb{E}\left[\phi\left(Z^{s,\cdot}_t\right)\right],\quad 0\le s \le t<T,\, \phi\in\mathcal{B}_b\left(\mathbb{R}^d\right),
	\end{equation}
	where $\mathcal{B}_b\left(\mathbb{R}^N\right)$ denotes the space of real--valued, Borel measurable and bounded functions defined on $\mathbb{R}^N$. The Chapman--Kolmogorov equations ensure that 
	\begin{equation}\label{CK}
		R_{s,t}\left(R_{t,u}\phi\right)={R_{s,u}\phi},\quad 0\le s<t<u\le T,\,\phi\in\mathcal{B}_b\left(\mathbb{R}^N\right).
	\end{equation}
	
	For every $0\le s <t\le T$ we define
	$
	F_{s,t} =\int_{s}^{t}e^{\left(t-r\right)A}f\left(r\right)dr\in \mathbb{R}^N$ and $I^L_{s,t}= \int_{s}^{t}e^{2\left(t-r\right)A}Q\,dL_r\colon\Omega\to\mathbb{R}^{N\times N}.
	$
	An adaptation of \cite[Theorem $6$]{AB}  guarantees that, for every $\phi\in\mathcal{B}_b\left(\mathbb{R}^N\right)$, the function $R_{s,t}\phi$ is differentiable at any point  $x\in\mathbb{R}^N$ in every direction $h\in\mathbb{R}^N$, with   
	\begin{equation}\label{no_bel}
		\left\langle\nabla^\top R_{s,t}\phi\left(x\right),h\right\rangle
		=
		\mathbb{E}\left[\phi\left(Z_{t}^{s,x}\right)
		\left\langle \left(I_{s,t}^L\right)^{-1}e^{\left(t-s\right)A}h,Z^{s,x}_{t}-e^{\left(t-s\right)A}x-F_{s,t}\right\rangle\right].
	\end{equation}
	Moreover, $R_{s,t}\phi\in C^1_b\left(\mathbb{R}^N\right)$ and the following gradient estimate holds true for some constant $c_\alpha>0$:
	\begin{equation}\label{est}
		\norm{\nabla^\top R_{s,t}\phi}_\infty
		\le c_\alpha\norm{\phi}_{\infty}\sup_{n=1,\dots, N}\left(\frac{1}{\sigma_n}\sqrt[2\alpha]{\frac{2\alpha\lambda_n}{1-e^{-2\alpha\lambda_n\left(t-s\right)}}}e^{-\lambda_n\left(t-s\right)}\right),\quad 0\le s < t\le T.
	\end{equation}
	In the sequel, for every $x\in\mathbb{R}^N$ and $t\in\left(0,T\right]$ we are going to need the continuity of $R_{\cdot,t}\phi\left(x\right)$ in the interval $\left[0,t\right)$ [resp., in the closed interval $\left[0,t\right]$]  when $\phi\in\mathcal{B}_b\left(\mathbb{R}^N\right)$ [resp., $\phi\in C_b\left(\mathbb{R}^N\right)$]. 
	In order to prove this property, we first note that  a variation of constants formula lets us consider (from \eqref{OU})
	\begin{equation}\label{explica}
	Z^{s,x}_t=e^{\left(t-s\right)A}x+\int_{s}^{t}e^{\left(t-r\right)A}f\left(r\right)dr+\int_{s}^{t}e^{\left(t-r\right)A}\sqrt{Q}\,dW_{L_r},\quad 0\le s\le t\le T,\quad x\in\mathbb{R}^N.
	\end{equation}
	This expression shows that the process $\left(Z^{s,x}_t\right)_{s\in\left[0,t\right]}$ is stochastically continuous (in the variable $s$). As a consequence, if  $\phi\in C_b\left(\mathbb{R}^N\right)$, then we  can easily deduce the continuity of $R_{\cdot,t}\phi\left(x\right)$ in $\left[0,t\right]$ applying the {continuous mapping} and {Vitali's convergence} theorems to \eqref{mse}. 
	In the general case $\phi\in\mathcal{B}_b\left(\mathbb{R}^N\right),$ one can use the same argument combined with the regularizing property of $R$ and \eqref{CK} to obtain the continuity of $R_{\cdot,t}\phi\left(x\right)$ in $\left[0,t\right)$, as desired.
	Finally, observe that there exists a constant $C=C\left(\alpha,A,Q\right)>0$ 
	such that 
	\[
	c_\alpha\sup_{n=1,\dots, N}\left(\frac{1}{\sigma_n}\sqrt[2\alpha]{\frac{2\alpha\lambda_n}{1-e^{-2\alpha\lambda_n\left(t-s\right)}}}e^{-\lambda_n\left(t-s\right)}\right)
	\le C\frac{1}{\left(t-s\right)^{1/\left(2\alpha\right)}},\quad 0\le s <t\le T.
	\]
	We refer to \cite[Remark $5$]{AB} for a similar computation. 
	Let us assume $\alpha\in\left(\frac{1}{2},1\right)$: in this way, denoting by $\gamma=1/\left(2\alpha\right)$, we have $\gamma\in\left(0,1\right)$ and the bound in \eqref{est} entails
	\begin{equation}\label{est12}
		\norm{\nabla^\top R_{s,t}\phi}_\infty\le C\norm{\phi}_\infty\frac{1}{\left(t-s\right)^\gamma},\quad 0\le s <t\le T,\, \phi\in\mathcal{B}_b\left(\mathbb{R}^N\right).
	\end{equation}
	
	For a given measurable and bounded vector field $B\colon \left[0,T\right]\times \mathbb{R}^N\to  \mathbb{R}^N$, we are concerned with the  analysis of the following  \emph{Kolmogorov backward equation} in mild, integral form:
	\begin{equation}\label{Kolm_mild}
		u^{\phi}_s\left(t,x\right)=
		R_{s,t}\phi\left(x\right)+\int_{s}^{t} R_{s,r}\left(\left\langle B\left(r,\cdot\right),\nabla^\top u^\phi_r\left(t,\cdot\right)\right\rangle\right) \left(x\right)dr,\quad s\in\left[0,t\right],\,x\in\mathbb{R}^N,
	\end{equation}
where $t\in\left(0,T\right]$ and $\phi\in\mathcal{B}_b\left(\mathbb{R}^N\right)$.
	We denote by $\norm{B}_{0,T}=\sup_{0\le t\le T}\norm{B\left(t,\cdot\right)}_\infty$. 
	In order to study \eqref{Kolm_mild}, for every $0< t_1<t_2\le T,$ we consider the Banach space $\left(\Lambda_1^\gamma\left[t_1,t_2\right],\norm{\cdot}_{\Lambda^\gamma_1{\left[t_1,t_2\right]}}\right)$ defined by
	\begin{align*}
		&\Lambda_1^\gamma{\left[t_1,t_2\right]}=\Big\{V\colon \left[t_1,t_2\right]\times \mathbb{R}^N\to\mathbb{R} \text{ measurable}:  V\left(\cdot,x\right)\in C\left(\left[t_1,t_2\right]\right),\,x\in\mathbb{R}^N;\\
		&\qquad\qquad\qquad\qquad\qquad\qquad\qquad\qquad\qquad\quad		 V\left(s,\cdot\right)\in C_b^1\left(\mathbb{R}^N\right),\,s\in\left[t_1,t_2\right];\,\sup_{s\in\left[t_1,t_2\right]}s^\gamma\norm{V\left(s,\cdot\right)}_1<\infty\Big\},\\ &\norm{V}_{\Lambda^\gamma_1{\left[t_1,t_2\right]}}= \sup_{s\in\left[t_1,t_2\right]}s^\gamma\norm{V\left(s,\cdot\right)}_1,\text{ where }\norm{V\left(s,\cdot\right)}_1=\norm{V\left(s,\cdot\right)}_\infty+\norm{\partial^1 V\left(s,\cdot\right)}_\infty.
	\end{align*}
	When $t_1=0$, we are careful to remove the left--end point of the interval $\left[t_1,t_2\right]$ in the previous definitions, so that we will be working with the space $\left(\Lambda^\gamma_1\left(0,t_2\right],\norm{\cdot}_{\Lambda^\gamma_1\left(0,t_2\right]}\right)$.
	The following lemma proves the well--posedness of \eqref{Kolm_mild}. We refer to \cite[Theorem $9.24$]{DPZ} for an analogous result concerning  the \emph{Kolmogorov forward equation}    in mild form associated with OU processes in infinite dimension corresponding to Brownian motions.
	\begin{theorem}\label{well_pos}
		Let $\alpha\in\left(\frac{1}{2},1\right)$ and $B\colon\left[0,T\right]\times \mathbb{R}^N\to \mathbb{R}^N$ be a measurable and bounded vector field. Then for every $\phi\in\mathcal{B}_b\left(\mathbb{R}^N\right)$ and $0<t\le T$, there exists a unique solution $u^\phi_s\left(t,x\right),\,s\in\left[0,t\right],\,x\in\mathbb{R}^N,$ of \eqref{Kolm_mild} such that $
		u^{\phi}_{t-\diamond}\left(t,\cdot\right)\in\Lambda^\gamma_1\left(0,t\right]$, where $\gamma=1/\left(2\alpha\right)$.
	\end{theorem}
	\begin{proof}
		Let us fix $\phi\in\mathcal{B}_b\left(\mathbb{R}^N\right),\,t\in\left(0,T\right],\,\bar{s}\in\left(0,t\right]$ and introduce the map $\Gamma_1\colon \Lambda_1^\gamma\left(0,\bar{s}\right]\to \Lambda_1^\gamma{\left(0,\bar{s}\right]}$ given by 
		\begin{equation}\label{Gamma}
			\Gamma_1 V\left(s,x\right)=  R_{t-s,t}\phi\left(x\right)+\int_{t-s}^{t} R_{t-s,r}\left(\left\langle B\left(r,\cdot\right), \nabla^\top V\left(t-r,\cdot\right)\right\rangle\right) \left(x\right)dr,\quad 0< s\le \bar{s},\,x\in\mathbb{R}^N,
		\end{equation}
		for every $V\in \Lambda^\gamma_1\left(0,{\bar{s}}\right]$. Notice that such an application is well defined and with values in $\Lambda_1^\gamma\left(0,{\bar{s}}\right]$, thanks to the properties of $R$ discussed above, the {dominated convergence theorem} and  the next computations based on \eqref{est12}:
		\begin{multline}\label{frac_est}
			\!\!\!\sup_{x\in\mathbb{R}^N}\!\left|\int_{t-s}^{t} \partial_{x_j} R_{t-s,r}\left(\left\langle B\left(r,\cdot\right), \nabla^\top V\left(t-r,\cdot\right)\right\rangle\right) \!\left(x\right)dr\right|
			\le\! N
			C \norm{B}_{0,T}
			\norm{V}_{\Lambda^\gamma_1\left(0,{\bar{s}}\right]}
			\int_{t-s}^{t}\frac{dr}{\left(r-\left(t-s\right)\right)^\gamma\left(t-r\right)^\gamma}
			\\\le
			\frac{4^\gamma}{1-\gamma} NC_{}\norm{B}_{0,T}\norm{V}_{\Lambda_1^\gamma\left(0,{\bar{s}}\right]}s^{1-2\gamma},\quad 0< s \le\bar{s},\,j=1,\dots,N.
		\end{multline}
		Here $C=C\left(\alpha,A,Q\right)>0$ is the same constant as in \eqref{est12}, and the last inequality is obtained using the  bound
		\begin{multline}\label{integrala}
			\int_{t-s}^{t}\frac{dr}{\left(r-\left(t-s\right)\right)^\gamma\left(t-r\right)^\gamma}
			=\left\{\int_{t-s}^{t-\frac{s}{2}}+\int_{t-\frac{s}{2}}^{t}\right\}\frac{dr}{\left(r-\left(t-s\right)\right)^\gamma\left(t-r\right)^\gamma}
			=2\int_{t-s}^{t-\frac{s}{2}}\frac{dr}{\left(r-\left(t-s\right)\right)^\gamma\left(t-r\right)^\gamma}\\
			\le\frac{2}{1-\gamma}\left(\frac{2}{s}\right)^\gamma\left(\frac{s}{2}\right)^{1-\gamma}=\frac{4^\gamma}{1-\gamma}s^{1-2\gamma},
		\end{multline}
		where for the second equality we perform the substitution $u=2t-s-r$.
		Estimates similar to those in \eqref{frac_est} allow to write, for every  $V_1,\,V_2\in\Lambda^\gamma_1\left(0,{\bar{s}}\right]$,
		\begin{multline*}
			\sup_{x\in\mathbb{R}^N}\left|\left(\Gamma_1V_1-\Gamma_1V_2\right)\left(s,x\right)\right|+ \sup_{x\in\mathbb{R}^N}\left|\partial_{x_j}\left(\Gamma_1V_1-\Gamma_1V_2\right)\left(s,x\right)\right|
			\\\le\frac{4^\gamma}{1-\gamma}N
			\norm{B}_{0,T}\left(s^{1-\gamma}+ C_{} s^{1-2\gamma}\right)\norm{V_1-V_2}_{\Lambda_1^\gamma\left(0,{\bar{s}}\right]},\quad 0< s \le \bar{s},\,j=1,\dots,N.
		\end{multline*}
		Hence we obtain
		\begin{equation}\label{contr}
			\norm{\Gamma_1V_1-\Gamma_1V_2}_{\Lambda_1^\gamma\left(0,{\bar{s}}\right]}\le\left[\frac{4^\gamma}{1-\gamma}N
			\norm{B}_{0,T} \left({\bar{s}}+C_{}{\bar{s}}^{1-\gamma}\right)\right] \norm{V_1-V_2}_{\Lambda^\gamma_1\left(0,{\bar{s}}\right]}.
		\end{equation}
		This shows that, for ${\bar{s}}$ sufficiently small, the map $\Gamma_1$ is a contraction in $\Lambda^\gamma_1\left(0,{\bar{s}}\right]$: we denote by $\widebar{V}_1$ its unique fixed point.  Now define 
		\begin{equation}\label{u}
			u^\phi_s\left(t,x\right)= R_{s,t}\phi\left(x\right)+\int_{s}^{t}R_{s,r}\left(\left\langle B\left(r,\cdot\right), \nabla^\top\widebar{V}_1\left(t-r,\cdot\right)\right\rangle\right)\left(x\right)dr,\quad t-{\bar{s}}\le s\le t,\,x\in\mathbb{R}^N,
		\end{equation}
		and notice that $ u_{t-s}^\phi\left(t,x\right)=\widebar{V}_1\left(s,x\right),\,0< s \le \bar{s} ,\,x\in\mathbb{R}^N.$
		Therefore $u_\diamond^\phi\left(t,\cdot\right)$ is the unique, local  solution of \eqref{Kolm_mild} (in the strip $\left[t-\bar{s},t\right]\times \mathbb{R}^N$) such that  $u_{t-\diamond}^\phi\left(t,\cdot\right)\in\Lambda^\gamma_1\left(0,{\bar{s}}\right].$
		
		At this point, we can repeat the same procedure to construct the solution of \eqref{Kolm_mild} in the interval $\left[t-2\bar{s},t-{\bar{s}}\right]$, because the relation among constants in \eqref{contr} --which is necessary to get a contraction-- does not depend on the initial condition. Specifically, we take  ${\phi}_1= u_{t-\bar{s}}^\phi\left(t,\cdot\right)\in C_b^1\left(\mathbb{R}^N\right)$  and define the map
		\begin{equation*}
			\Gamma_2 V\left(s,x\right)=  R_{t-s,t-\bar{s}}\,{\phi}_1\left(x\right)+\int_{t-s}^{t-\bar{s}} R_{t-s,r}\left(\left\langle B\left(r,\cdot\right), \nabla^\top V\left(t-r,\cdot\right)\right\rangle\right) \left(x\right)dr,\quad \bar{s}\le s\le 2\bar{s},\,x\in\mathbb{R}^N,
		\end{equation*} 
		for every $V\in\Lambda_1^\gamma\left[\bar{s},2\bar{s}\right]$. Computations analogous to the ones in the previous step show that $\Gamma_2\colon\Lambda^\gamma_1\left[\bar{s},2\bar{s}\right]\to\Lambda^\gamma_1\left[\bar{s},2\bar{s}\right]$ is a contraction: its unique fixed point is denoted by $\widebar{V}_2$. Then we call
		\[
		u_s^{{\phi}_1}\left(t-\bar{s},x\right)= R_{s,t-\bar{s}}{\phi}_1\left(x\right)+\int_{s}^{t-\bar{s}}R_{s,r}\left(\left\langle B\left(r,\cdot\right), \nabla^\top \widebar{V}_2\left(t-r,\cdot\right)\right\rangle\right)\left(x\right)dr,\quad t-2{\bar{s}}\le s\le t-{\bar{s}},\,x\in\mathbb{R}^N;
		\]
		notice that $u^{{\phi}_1}_{t-{s}}\left(t-\bar{s},x\right)=\widebar{V}_2\left(s,x\right),\,\bar{s}\le s\le2\bar{s},\,x\in\mathbb{R}^N$, and that by the definition of ${\phi}_1$, one has $u^{{\phi}_1}_{t-\bar{s}}\left(t-\bar{s},\cdot\right)=u^{\phi}_{t-\bar{s}}\left(t,\cdot\right)$. Now we extend the function $u^\phi_s\left(t,x\right)$ in \eqref{u} assigning
		\[
		\widebar{u}_s^\phi\left(t,x\right)= \begin{cases}
			u^\phi_s\left(t,x\right),&t-\bar{s}\le s\le t\\ 
			u_s^{{\phi}_1}\left(t-\bar{s},x\right),&t-2\bar{s}\le s\le t-\bar{s}
		\end{cases},\quad x\in\mathbb{R}^N.
		\]
		By the Chapman--Kolmogorov equations and Fubini's theorem we realize that $\widebar{u}^\phi_\diamond\left(t,\cdot\right)$ is the unique local solution of \eqref{Kolm_mild} (in the strip $\left[t-2\bar{s},t\right]\times \mathbb{R}^N$) such that $\widebar{u}^\phi_{t-\diamond}\left(t,\cdot\right)\in\Lambda_1^\gamma\left(0,2\bar{s}\right].$ In the sequel, we can simply denote it by $u^\phi_\diamond\left(t,\cdot\right)$.
		
		This argument by steps of lenght $\bar{s}$ can be repeated iteratively to cover the whole interval $\left[0,t\right]$ and obtain the unique, global solution $u^\phi_\diamond\left(t,\cdot\right)$ of \eqref{Kolm_mild} such that $ u^{\phi}_{t-\diamond}\left(t,\cdot\right)\in\Lambda^\gamma_1\left(0,t\right]$. Thus, the proof is complete.
	\end{proof}
	If $\phi\in C_b^1\left(\mathbb{R}^N\right)$, then recalling \eqref{explica} one can directly write  $\nabla R_{s,t}\phi\left(x\right)=\mathbb{E}\left[\nabla\phi\left(Z^{s,x}_t\right)\right]e^{\left(t-s\right)A}$. Next, considering that $\left|e^{\left(t-s\right)A}\right|\le 1, \, 0\le s \le t\le T$, an application of \eqref{no_bel}-\eqref{est12} shows that $R_{s,t}\phi\in C_b^2\left(\mathbb{R}^N\right),$ with 
	\begin{equation*}
		\norm{\partial^2R_{s,t}\phi}_\infty\le C_{}\norm{\partial^1 \phi}_\infty\frac{1}{\left(t-s\right)^\gamma},\quad 0\le s <t\le T,
	\end{equation*}
where $C=C\left(\alpha,A,Q\right)>0$ is the same constant as in \eqref{est12}.  This argument can be iterated to claim that, given an integer $n\ge2$ and $\phi\in C_b^{n-1}\left(\mathbb{R}^N\right)$, $R_{s,t}\phi\in C_b^n\left(\mathbb{R}^N\right)$ and 
\begin{equation}\label{he}
		\norm{\partial^{n}R_{s,t}\phi}_\infty\le C_{}\norm{\partial^{n-1} \phi}_\infty\frac{1}{\left(t-s\right)^\gamma},\quad 0\le s <t\le T.
\end{equation}
	The previous consideration allows to extend Lemma \ref{well_pos}. To this purpose,  for an integer $n\ge 2$ and  $0<t_1<t_2\le T$ we introduce the Banach space $\left(\Lambda_n^\gamma{\left[t_1,t_2\right]}, \norm{V}_{\Lambda^\gamma_n{\left[t_1,t_2\right]}}\right)$ defined by 
	\begin{align*}
		&\Lambda_n^\gamma{\left[t_1,t_2\right]}=\Big\{V\colon \left[t_1,t_2\right]\times \mathbb{R}^N\to\mathbb{R} \text{ measurable}:  V\left(\cdot,x\right)\in C\left(\left[t_1,t_2\right]\right),\,x\in\mathbb{R}^N;\\
		&\qquad\qquad\qquad\qquad\qquad\qquad\qquad\qquad\qquad\quad		
		 V\left(s,\cdot\right)\in C_b^n\left(\mathbb{R}^N\right),\,s\in\left[t_1,t_2\right];\,\sup_{s\in\left[t_1,t_2\right]}s^\gamma\norm{V\left(s,\cdot\right)}_n<\infty\Big\},\\
		&\norm{V}_{\Lambda^\gamma_n{\left[t_1,t_2\right]}}= \sup_{s\in\left[t_1,t_2\right]}s^\gamma\norm{V\left(s,\cdot\right)}_n,\text{ where }\norm{V\left(s,\cdot\right)}_n=\norm{V\left(s,\cdot\right)}_{\infty}+\sum_{j=1}^{n}\norm{\partial^j V\left(s,\cdot\right)}_\infty.
	\end{align*}
	As we have done before, when $t_1=0$ we remove the left--end point of $\left[t_1,t_2\right].$
	\begin{corollary}\label{cor2}
		Let $\alpha\in\left(\frac{1}{2},1\right),\,n\ge 2$ be an integer and $B\in C_b^{0,n-1}\left(\left[0,T\right]\times \mathbb{R}^N;\mathbb{R}^N\right)$. Then for every $\phi\in{C}_b^{n-1}\left(\mathbb{R}^N\right)$ and $0<t\le T$, there exists a unique solution $u^\phi_s\left(t,x\right),\,s\in\left[0,t\right],\,x\in\mathbb{R}^N,$ of \eqref{Kolm_mild} such that $
		u^{\phi}_{t-\diamond}\left(t,\cdot\right)\in\Lambda^\gamma_n\left(0,t\right],$ where $\gamma=1/\left(2\alpha\right)$.
	\end{corollary}
	\begin{proof}
		Take an integer $n\ge 2$; the argument parallels the one in the proof of Lemma \ref{well_pos}, so here we only show that, for a given $\phi\in C_b^{n-1}\left(\mathbb{R}^N\right)$ and $\bar{s}\in\left(0,t\right]$ sufficiently small, the map $\Gamma_1\colon\Lambda_n^\gamma\left(0,\bar{s}\right]\to\Lambda_n^\gamma\left(0,\bar{s}\right]$  in \eqref{Gamma} is well defined and a contraction. First, we note that for every $V\in\Lambda_n^\gamma\left(0,\bar{s}\right]$ and multi--index $\mathbf{j}$ such that $1\le \norm{\mathbf{j}}_1\le n$,
		\[
		\partial_{\mathbf{j}}\Gamma_1V\left(s,x\right)= \partial_{\mathbf{j}} R_{t-s,t}\phi\left(x\right)+\int_{t-s}^{t} \partial_{\mathbf{j}}R_{t-s,r}\left(\left\langle B\left(r,\cdot\right), \nabla^\top V\left(t-r,\cdot\right)\right\rangle\right) \left(x\right)dr,\quad 0< s\le \bar{s},\,x\in\mathbb{R}^N,
		\]
		and that $\sup_{s\in\left(0,\bar{s}\right]}s^\gamma \norm{\partial^{\norm{\mathbf{j}}_1}R_{t-s,t}\phi}_\infty<\infty$ by \eqref{he}. Secondly,  invoking the estimates in \eqref{integrala} and \eqref{he}, for every $ 0< s \le\bar{s},$
		\begin{multline*}
			\sup_{x\in\mathbb{R}^N}\left|\int_{t-s}^{t} \partial_{\mathbf{j}}R_{t-s,r}\left(\left\langle B\left(r,\cdot\right), \nabla^\top V\left(t-r,\cdot\right)\right\rangle\right) \left(x\right)dr\right|\\ 
			\le N C_nC\norm{B}_{n-1,T}
			\norm{V}_{\Lambda^\gamma_n\left(0,{\bar{s}}\right]}
			\int_{t-s}^{t}\frac{dr}{\left(r-\left(t-s\right)\right)^\gamma\left(t-r\right)^\gamma}\\
			\le
			\frac{4^\gamma}{1-\gamma} NC_n C_{}\norm{B}_{n-1,T}\norm{V}_{\Lambda_n^\gamma\left(0,{\bar{s}}\right]}s^{1-2\gamma},\quad C_n=\binom{n-1}{\left[2^{-1}\left(n-1\right)\right]},
		\end{multline*}
		where $\norm{B}_{n-1,T}=\sup_{0\le t\le T}\left(\norm{B\left(t,\cdot\right)}_\infty+\sum_{j=1}^{n-1}\norm{\partial^jB\left(t,\cdot\right)}_\infty\right)$ and $C=C\left(\alpha,A,Q\right)>0$ is the same constant as in \eqref{est12}. It then follows that $\Gamma_1V\in\Lambda_n^\gamma\left(0,\bar{s}\right]$, with 
		\begin{equation*}
			\norm{\Gamma_1V_1-\Gamma_1V_2}_{\Lambda_n^\gamma\left(0,{\bar{s}}\right]}\le\left[\frac{4^\gamma}{1-\gamma}N
			\norm{B}_{n-1,T} \left(\bar{s}+ nC_nC_{}{\bar{s}}^{1-\gamma}\right)\right] \norm{V_1-V_2}_{\Lambda^\gamma_n\left(0,{\bar{s}}\right]},\quad V_1,V_2\in\Lambda_n^\gamma\left(0,{\bar{s}}\right],
		\end{equation*}
		which reduces to \eqref{contr} when $n=1$ and proves the contraction property of $\Gamma_1$ for $\bar{s}$ small enough.
	\end{proof}
\section{The time--dependent Markov transition semigroup}\label{a_conn}
	Let $\alpha\in\left(0,1\right)$ and introduce a vector field $B_0\colon \left[0,T\right]\times \mathbb{R}^N\to \mathbb{R}^N$ such that $B_0\in C_b^{0,1}\left(\left[0,T\right]\times \mathbb{R}^N;\mathbb{R}^N\right)$. For every $x\in\mathbb{R}^N$ and $0\le s \le T$, we define the process $X^{s,x}=\left(X^{s,x}_t\right)_{t\in\left[s,T\right]}$ to be  the unique (up to indistinguishability) solution of the semilinear stochastic differential equation 
	\begin{equation}\label{st_semi}
		dX^{s,x}_t=\left(AX_{t}^{s,x}+B_0\left(t,X_t^{s,x}\right)\right)dt+\sqrt{Q}\,dW_{L_t},\quad X_s^{s,x}=x\in\mathbb{R}^N.
	\end{equation}
	We denote by $P=\left(P_{s,t}\right),\,0\le s \le t \le T$, the corresponding time--dependent Markov transition semigroup given by
	\[
	P_{s,t}\phi=\mathbb{E}\left[\phi\left(X^{s,\cdot}_t\right)\right],\quad \phi\in\mathcal{B}_b\left(\mathbb{R}^N\right).
	\]
	The connection between the SDE in \eqref{st_semi} and the Kolmogorov backward equation in mild integral form \eqref{Kolm_mild}  is provided by the next, fundamental result.
	\begin{theorem}\label{connection}
		Let $\alpha\in\left(\frac{1}{2},1\right),\,B_0\in C_b^{0,3}\left(\left[0,T\right]\times \mathbb{R}^N;\mathbb{R}^N\right),\,f\in C\left(\left[0,T\right];\mathbb{R}^N\right)$ and define $B= B_0-f$. Then, for every $\phi\in\mathcal{B}_b\left(\mathbb{R}^N\right)$ and $0<t\le T$, the function $P_{s,t}\phi\left(x\right),\,0\le s \le t,\,x\in\mathbb{R}^N$, is the unique solution of \eqref{Kolm_mild} such that $ P_{t-\diamond,t}\phi\left(\cdot\right)\in\Lambda_1^\gamma\left(0,t\right],$ where $\gamma=1/\left(2\alpha\right)$.
	\end{theorem}
	The purpose of this section is to develop a self--contained procedure which is specific to our framework and allows to prove Theorem \ref{connection} via important, preliminary results.  In the case of time--independent nonlinearities and $f\equiv0$ (hence for  Kolmogorov  forward equations in mild form), Theorem \ref{connection} is known for noises different from our $W_L$. As regards independent $\alpha-$stable Lévy processes in finite dimension, it has been established in \cite[Lemma $5.12$]{PZ} (its proof relies on the theory of one--parameter semigroups, so it cannot be adapted to our framework). As for Brownian motions in infinite dimension, we refer to \cite[Theorem $9.27$]{DPZ}.
	
	Let $\alpha\in\left(0,1\right), \,B_0\in C_b^{0,1}\left(\left[0,T\right]\times \mathbb{R}^N; \mathbb{R}^N\right)$ and recall that the subordinated Brownian motion $W_L$ is an isotropic (i.e., rotation--invariant), $2\alpha-$stable, $\mathbb{R}^N-$valued Lévy process with compensator $\nu\left(dz\right)\asymp \left|z\right|^{-\left(N+2\alpha\right)}dz$ and no continuous martingale part (see \cite[Theorem $30.1$]{Sato}). Here $\asymp$ denotes the equality up to a positive multiplicative constant. By \cite[Theorem $3.1$]{priola} (see also \cite{pr1}) there is a sharp stochastic flow $X^{s,x}_t$ generated by the SDE \eqref{st_semi} which is jointly measurable in $\left(s,t,x,\omega\right)$ and, $\mathbb{P}-$a.s., simultaneously continuous in $x$ and càdlàg in $s$ and $t$. More specifically, there exists an almost--sure event $\Omega'$ such that the following facts hold true for every $\omega\in\Omega'$:
	\begin{itemize}
		\item for every $x\in\mathbb{R}^N$ and $t\in\left[0,T\right]$, the mapping $s\mapsto X^{s,x}_t\left(\omega\right)$ is càdlàg in $\left[0,t\right]$;
		\item for every $x\in\mathbb{R}^N$ and  $s\in\left[0,T\right]$, the mapping $t\mapsto X^{s,x}_t\left(\omega\right)$ is càdlàg in $\left[s,T\right]$;
		\item for every $0\le s\le t\le T$, the mapping $x\mapsto X^{s,x}_t\left(\omega\right)$ is continuous in $\mathbb{R}^N$;
		\item the flow property is satisfied, namely $X^{s,x}_t\left(\omega\right)=X_t^{r,X^{s,x}_r\left(\omega\right)}\left(\omega\right)$ for every $x\in\mathbb{R}^N,\,0\le s <r<t\le T$;
		\item for every $x\in\mathbb{R}^N$ and $0\le s \le t\le T$, $X^{s,x}_t\left(\omega\right)=x+\int_{s}^{t}\left(AX^{s,x}_r\left(\omega\right)+B_0\left(r,X^{s,x}_r\left(\omega\right)\right)\right)dr+ \sqrt{Q}\left(W_{L_t}-W_{L_s}\right)\left(\omega\right)$.
	\end{itemize}
	For every $\omega\in\Omega\setminus\Omega'$, we set $X^{s,x}_t\left(\omega\right)=x,\, \left(s,\,t\right) \in\left[0,T\right]^2,\,x\in\mathbb{R}^N$: from now on, we work with such a stochastic flow $X^{s,x}_t$. The next result shows that, under additional regularity requirements on $B_0$, it is differentiable with respect to $x$. Analogous claims concerning differentiability of stochastic flows can be found in literature in, e.g.,  \cite[Theorem $8.18$]{dps} for the Brownian case and in \cite[Theorem $3.4.2$]{Ku19} for the jumps one, although the latter requires regularity assumptions on the coefficients which are not fulfilled by our framework. The proof, which carries out a path--by--path argument thanks to the already mentioned properties guaranteed by \cite{priola}, is postponed to \ref{appendix}.
	\begin{lemma}\label{hoc}
		Let $\alpha\in\left(0,1\right),\, n\ge 2$ be an integer and $B_0\in C^{0,n}_b\left(\left[0,T\right]\times \mathbb{R}^N;\mathbb{R}^N\right)$. Then for every $\omega\in\Omega$ and $0\le s\le t\le T$, the function $x\mapsto X^{s,x}_t\left(\omega\right)$ belongs to $C^n\left(\mathbb{R}^N\right)$,
		and there exists a constant $C>0$ depending only on $A,\,B_0,\,T,\,n$ and $N$ such that
		\begin{equation}\label{boundala}
			\sum_{i=1}^{n}\norm{\partial^iX^{s,\cdot}_t\left(\omega\right)}_\infty\le C,\quad 0\le s \le t \le T,\,\omega\in\Omega.
		\end{equation}
	\end{lemma}
	The previous  claim implies the following result regarding persistence of regularity.
	\begin{corollary}\label{higher}
		Let $\alpha\in\left(0,1\right),\,n\ge2$ be an integer  and $\phi\in C_b^n\left(\mathbb{R}^N\right)$. If $B_0\in C_b^{0,n}\left(\left[0,T\right]\times \mathbb{R}^N;\mathbb{R}^N\right)$, then for every $0\le s\le t\le T$ the function $P_{s,t}\phi\in C^n_b\left(\mathbb{R}^N\right)$. In addition, 
		\begin{equation}\label{boundala_meglio}
			\sup_{0\le s \le t \le T}\left(\norm{P_{s,t}\phi}_\infty+\sum_{i=1}^n\norm{\partial^i P_{s,t}\phi}_\infty\right)<\infty.
		\end{equation}
	\end{corollary}
	Let $D=\left\{z\in\mathbb{R}^N,\, \left|z\right|\le 1 \right\}$; we introduce the family of integro--differential operators $\left(A\left(s\right)\right)_{0\le s\le T}$, defined on every $\psi\in C_b^2\left(\mathbb{R}^N\right)$ by
	\begin{equation}\label{operator}
		A\left(s\right)\psi\left(x\right)=\left\langle Ax+B_0\left(s,x\right), \nabla^\top \psi \left(x\right)\right\rangle + \int_{\mathbb{R}^N}\left[\psi\left(x+\sqrt{Q}z\right)-\psi\left(x\right)-1_D\left(z\right)\nabla \psi\left(x\right)\sqrt{Q}z\right]\nu\left(dz\right),
	\end{equation}
	where $x\in\mathbb{R}^N$. We need the next preparatory result.
	\begin{lemma}\label{fond}
		\begin{enumerate}[label=$\left(\emph{\roman*}\right)$]
			\item\label{poi1} Let $\alpha\in\left(\frac{1}{2},1\right),\,0\le s \le T$ and $x\in\mathbb{R}^N$. If $B_0\in C^{0,1}_b\left(\left[0,T\right]\times \mathbb{R}^N; \mathbb{R}^N\right)$, then  the mapping $r\mapsto P_{s,r}A\left(r\right)\phi\left(x\right)$ is continuous in $\left[s,T\right]$ for every $\phi\in C^2_b\left(\mathbb{R}^N\right)$;
			\item \label{poi2} Let $\alpha\in\left(0,1\right)$ and $0\le t \le T$. If $B_0\in C^{0,3}_b\left(\left[0,T\right]\times \mathbb{R}^N; \mathbb{R}^N\right)$, then for every $r\in\left[0,t\right]$ and $\phi\in C_b^3\left(\mathbb{R}^N\right)$ the mapping $x\mapsto A\left(r\right)P_{r,t}\phi\left(x\right)$ belongs to $C^1\left(\mathbb{R}^N\right)$. Moreover, $\sup_{ r\in\left[0,t\right]}\norm{1_B\nabla^\top A\left(r\right)P_{r,t}\phi}_\infty<\infty$ for every bounded set $B\subset \mathbb{R}^N$.
		\end{enumerate}
	\end{lemma}
	\begin{proof}	
		We start off by proving Point \emph{\ref{poi1}}. Fix $0\le s\le T$ and $x\in\mathbb{R}^N$; from \eqref{st_semi}, Gronwall's lemma, \cite[Theorem $3.2$]{MR} and the continuity in probability of the Lévy process $W_L$ we deduce that $\mathbb{E}\left[\sup_{ t\in\left[s,T\right]}\left|X^{s,x}_t\right|^p\right]<\infty$ for every $p\in\left(1,2\alpha\right)$, and that the process $X_{\cdot}^{s,x}$ is stochastically continuous in $\left[s,T\right]$, as well. Consider $r\in\left[s,T\right]$ and a sequence $\left(r_n\right)_n\subset \left[s,T\right]$ such that $r_n\to r$ as $n\to \infty$. Given $\phi\in C^2_b\left(\mathbb{R}^N\right)$,
		\begin{multline*}
			P_{s,r_n}A\left(r_n\right)\phi\left(x\right)-P_{s,r}A\left(r\right)\phi\left(x\right)=
			P_{s,r_n}\left(A\left(r_n\right)\phi-A\left(r\right)\phi\right)\left(x\right)+
			\left(P_{s,r_n}A\left(r\right)\phi\left(x\right)-P_{s,r}A\left(r\right)\phi\left(x\right)\right)\\
			\eqqcolon\mathbf{\upperRomannumeral{1}}_n+\mathbf{\upperRomannumeral{2}}_n.
		\end{multline*}
		Since \eqref{operator} entails $\left(A\left(r_n\right)\phi-A\left(r\right)\phi\right)\left(\cdot\right)=\left\langle B_0\left(r_n,\cdot\right)-B_0\left(r,\cdot\right),\nabla^\top\phi\left(\cdot\right)\right\rangle$ we have, by Vitali's and dominated convergence theorems,
		\[
		\left|\mathbf{\upperRomannumeral{1}}_n\right|\le 
		\norm{\nabla^\top\phi}_\infty\left(2\norm{DB_0}_{T,\infty}\mathbb{E}\left[\left|X^{s,x}_{r_{n}}-X^{s,x}_{r}\right|\right]+\mathbb{E}\left[\left|B_0\left(r_n,X^{s,x}_r\right)-B_0\left(r,X^{s,x}_r\right)\right|\right]\right)\underset{n\to \infty}{\longrightarrow} 0,
		\]
	where we denote by $\norm{DB_{0}}_{T,\infty}=\sup_{ 0\le t \le T}\norm{DB_0\left(t,\cdot\right)}_\infty$. 
		As for $\mathbf{\upperRomannumeral{2}}_n$, note that $A\left(r\right)\phi$ is continuous in $\mathbb{R}^N$, and that for every $y\in\mathbb{R}^N$ (see \eqref{operator}),
		\begin{multline}\label{salvare}
			\left|A\left(r\right)\phi\left(y\right)\right|
			\le
			\norm{\nabla^\top\phi}_\infty\left(\left|A\right|\left|y\right|+\norm{B_0}_{0,T}\right)\\
			+\frac{1}{2}\norm{D^2\phi}_{\infty}\int_{\mathbb{R}^N}1_D\left(z\right)\left|\sqrt{Q}z\right|^2\nu\left(dz\right)+2\norm{\phi}_\infty\int_{\mathbb{R}^N}1_{D^c}\left(z\right)\nu\left(dz\right).
		\end{multline}
		Therefore by the continuous mapping and Vitali's convergence theorem we obtain $\mathbf{\upperRomannumeral{2}}_n\to0$ as $n\to\infty$, proving Point \emph{\ref{poi1}}. 
		
		We now move on to Point \emph{\ref{poi2}}, where it is sufficient to require $\alpha\in\left(0,1\right)$. Fix $0\le r \le t \le T$; observe that for every $\psi\in C_b^3\left(\mathbb{R}^N\right)$ one has $A\left(r\right)\psi\in C^1\left(\mathbb{R}^N\right)$, with 
		\begin{multline*}
			\nabla A\left(r\right)\psi\left(x\right)=
			\nabla\psi\left(x\right)
			\left(A+DB_0\left(r,x\right)\right)+\left(Ax+B_0\left(r,x\right)\right)^\top D^2\psi\left(x\right)\\
			+
			\int_{\mathbb{R}^N}\left[\nabla\psi\left(x+\sqrt{Q}z\right)-\nabla\psi\left(x\right)-1_D\left(z\right)\left(\sqrt{Q}z\right)^\top D^2\psi\left(x\right)\right]\nu\left(dz\right),\quad x\in\mathbb{R}^N.
		\end{multline*}
		More specifically, in the previous computation we are allowed to differentiate under the integral sign because
		\[
		\left|\nabla^\top \psi\left(x+\sqrt{Q}z\right)-\nabla^\top \psi\left(x\right)-D^2\psi\left(x\right)\sqrt{Q}z \right|\le \frac{1}{2}N^{\frac32}\norm{\partial^3\psi}_\infty\left|\sqrt{Q}z\right|^2,\quad x\in\mathbb{R}^N,\,z\in D.
		\]
		The hypotheses prescribe $B_0\in C^{0,3}_b\left(\left[0,T\right]\times \mathbb{R}^N; \mathbb{R}^N\right)$ and $\phi\in C^3_b\left(\mathbb{R}^N\right)$,  hence it is sufficient to invoke  Corollary \ref{higher} to complete proof.
	\end{proof}
	We are now in position to prove the following, crucial result concerning Kolmogorov  equations (cfr. \cite[Theorem~$4.5.1$]{Ku19} for an analogous claim in a different setting).
	\begin{theorem}\label{kfb}
		\begin{enumerate}[label=$\left(\emph{\roman*}\right)$]
			Take $\alpha\in\left(\frac{1}{2},1\right)$.
			\item\label{for_k} Let $0\le s \le T$ and $x\in\mathbb{R}^N$. If $B_0\in C^{0,1}_b\left(\left[0,T\right]\times \mathbb{R}^N; \mathbb{R}^N\right)$ and $\phi\in C_b^2\left(\mathbb{R}^N\right)$, then the function $t\mapsto P_{s,t}\phi\left(x\right)$ is continuously differentiable in $\left[s,T\right]$ and satisfies the Kolmogorov forward equation
			\begin{equation}\label{kol_for}
				\partial_t P_{s,t}\phi\left(x\right)=P_{s,t}A\left(t\right)\phi\left(x\right);
			\end{equation}
			\item \label{back_k}Let $0\le t \le T$ and  $x\in\mathbb{R}^N$. If $B_0\in C^{0,3}_b\left(\left[0,T\right]\times \mathbb{R}^N; \mathbb{R}^N\right)$ and $\phi\in C_b^3\left(\mathbb{R}^N\right)$, then the function $s\mapsto P_{s,t}\phi\left(x\right)$ is continuously differentiable in $\left[0,t\right]$ and satisfies the Kolmogorov backward equation
			\begin{equation}\label{kol_back}
				\partial_s P_{s,t}\phi\left(x\right)=-A\left(s\right)P_{s,t}\phi\left(x\right).
			\end{equation}
		\end{enumerate}
	\end{theorem}
	\begin{proof}
		Recall that by \cite[Theorem $14.7$ (\lowerRomannumeral{3})]{Sato} the process $W_L$ is centered in $0$ when $\alpha\in\left(\frac{1}{2},1\right)$. As a consequence, denoting by $N$ the Poisson random measure associated with its jumps and by $\widetilde{N}$ the compensated measure, $W_{L}=\int_{0}^{\cdot}\int_{\mathbb{R}^N}1_D\left(z\right)z\widetilde{N}\left(ds,dz\right)+\int_{0}^{\cdot}1_{D^c}\left(z\right)zN\left(ds,\,dz\right)$ up to indistinguishability by \cite[Theorem $2.34$, Chapter~ \upperRomannumeral{2}]{js}. 
		
		As for  Point \emph{\ref{for_k}}, take $0\le s \le T,\,x\in\mathbb{R}^N$ and $\phi\in C^2_b\left(\mathbb{R}^N\right)$; by \eqref{st_semi} an application of It\^o formula 
		ensures that
		\begin{multline*}
			\phi\left(X^{s,x}_t\right)=
			\phi\left(x\right)+\int_{s}^{t}\left\langle A X^{s,x}_r+B_0\left(r,X^{s,x}_r\right),\nabla^\top\phi\left(X_r^{s,x}\right)\right\rangle dr+\int_{s}^{t}\int_{\mathbb{R}^N}1_D\left(z\right)\nabla\phi\left(X^{s,x}_{r-}\right)\sqrt{Q}z\,\widetilde{N}\left(dr,dz\right)\\
			+
			\int_{s}^{t}\int_{\mathbb{R}^N}\left(\phi\left(X_{r-}^{s,x}+\sqrt{Q}z\right)-\phi\left(X_{r-}^{s,x}\right)-1_D\left(z\right)\nabla\phi\left(X^{s,x}_{r-}\right)\sqrt{Q}z\right){N}\left(dr,dz\right),
		\end{multline*}
		which holds true $\mathbb{P}-\text{a.s.}$ for every $t\in\left[s,T\right]$.	Taking expectations in the previous equation and using Fubini's theorem we obtain
		\[
		P_{s,t}\phi\left(x\right)=\phi\left(x\right)+\int_{s}^t\mathbb{E}\left[A\left(r\right)\phi\left(X^{s,x}_r\right)\right]dr=\phi\left(x\right)+\int_{s}^{t}P_{s,r}A\left(r\right)\phi\left(x\right)dr,\quad t\in\left[s,T\right],
		\]
		which in turn implies \eqref{kol_for} by Lemma \ref*{fond} \emph{\ref{poi1}}.
		
		We now focus on Point \emph{\ref{back_k}}. Take $0\le t \le T$ and $x\in\mathbb{R}^N$; arguing as in \cite[Proposition $3.8.2$]{Ku19} we see that $X^{\cdot,x}_t$ follows the backward dynamics ($\mathbb{P}-$a.s.) 
		\begin{multline*}
			X^{s,x}_t=x+\int_{s}^{t}DX^{r,x}_t\left(Ax+B_0\left(r,x\right)\right)dr+\int_{s}^{t}\int_{\mathbb{R}^N}\left[X^{r,x+\sqrt{Q}z}_t-X^{r,x}_t-1_D\left(z\right)DX^{r,x}_t\sqrt{Q}z\right]\nu\left(dz\right)dr\\
			+\int_{s}^{t}\int_{\mathbb{R}^N}\left(X^{r,x+\sqrt{Q}z}_t-X^{r,x}_t\right)\widetilde{N}\left(dr,dz\right),\quad s\in\left[0,t\right].
		\end{multline*}
		Hence invoking the backward It\^o formula (see, e.g., \cite[Theorem $2.7.1$]{Ku19}) we deduce that, for every $\phi\in C_b^2\left(\mathbb{R}^N\right)$ and $s\in\left[0,t\right]$,
		\begin{multline*}
			\!\!\!\!\phi\left(X^{s,x}_t\right)=\phi\left(x\right)+
			\int_{s}^{t}\!\int_{\mathbb{R}^N}\left(\phi\left(X_t^{r,x+\sqrt{Q}z}\right)-\phi\left(X^{r,x}_t\right)\right)\widetilde{N}\left(dr,dz\right)\!
			+\int_{s}^{t}\!\nabla \phi\left(X^{r,x}_t\right)DX^{r,x}_t\left(Ax+B_0\left(r,x\right)\right)dr
			\\+\int_{s}^{t}\int_{\mathbb{R}^N}\left[\phi\left(X_t^{r,x+\sqrt{Q}z}\right)-\phi\left(X^{r,x}_t\right)-1_D\left(z\right)\nabla \phi\left(X^{r,x}_t\right)DX^{r,x}_t\sqrt{Q}z\right]\nu\left(dz\right)dr,
		\end{multline*}
		which holds true $\mathbb{P}-$a.s.
		Taking expectations in the previous equation and using Fubini's theorem (remember Lemma \ref{hoc}) we obtain
		\begin{equation}\label{ste1}
			P_{s,t}\phi\left(x\right)=\phi\left(x\right)+\int_{s}^tA\left(r\right)P_{r,t}\phi\left(x\right)dr,\quad s\in\left[0,t\right].
		\end{equation}
		Since by hypotheses we are working with $\phi\in C_b^3\left(\mathbb{R}^N\right)$ and $B_0\in C^{0,3}_b\left(\left[0,T\right]\times \mathbb{R}^N; \mathbb{R}^N\right)$, by Lemma \ref*{fond} \emph{\ref{poi2}} we can differentiate in $x$ the expression in \eqref{ste1}, showing the continuity of the mapping $r\mapsto \nabla P_{r,t}\phi\left(x\right)$ in $\left[0,t\right]$. This, together with \eqref{operator}, the fact that \eqref{ste1} also provides the continuity of the mapping $r\mapsto P_{r,t}\phi\left(x\right)$ in $\left[0,t\right]$ and a dominated convergence argument based on Corollary \ref{higher}, ensures the continuity of the function $r\mapsto A\left(r\right)P_{r,t}\phi\left(x\right)$ in the same interval. Therefore differentiating  \eqref{ste1} with respect to $s$ we infer \eqref{kol_back}. The proof is now complete.
	\end{proof}
	
	Another  step that we need to prove Theorem \ref{connection} consists in a regularization result for the time--dependent Markov transition semigroup $P_{s,t}$ (see Lemma \ref{bou}) which --at the best of our knowledge-- is not established in literature with this type of noise. We start by recalling the \emph{Bismut--Elworthy--Li's type formula} presented in \cite[Theorem $1.1$]{Z} (see also \cite{Zm} for a related work treating multiplicative Lévy noise); such a formula is adapted to our framework, where we have to account for an initial time $s$ not necessarily equal to $0$.
	\begin{theorem}[\cite{Z}]\label{zhang}
		Let $\alpha\in\left(\frac{1}{2},1\right)$ and $B_0\in C_b^{0,1}\left(\left[0,T\right]\times\mathbb{R}^N ;\mathbb{R}^N\right)$. Then for every $0\le s<t\le T$ and $\phi\in C^1_b\left(\mathbb{R}^N\right)$,  the function $P_{s,t}\phi$ is differentiable at  $x$ in every direction $h\in\mathbb{R}^N$ and 
		\begin{equation}\label{bel}
			\left\langle\nabla^\top P_{s,t}\phi\left(x\right),h\right\rangle
			=
			\mathbb{E}\left[\frac{1}{L_t-L_s}\phi\left({X}_t^{s,x}\right)\int_{s}^{t}\left\langle\left(\sqrt{Q}\right)^{-1}D_h{X}^{s,x}_r, dW_{L_r}\right\rangle\right].
		\end{equation}
		
		Furthermore, there exists a constant ${C}_\alpha>0$ such that the next gradient estimate holds true:
		\begin{equation}\label{est_z}
			\norm{\nabla^\top P_{s,t}\phi}_\infty
			\le{C}_\alpha
			\norm{\phi}_{\infty}\left|\left(\sqrt{Q}\right)^{-1}\right|e^{\left(\left|A\right|+\norm{DB_0 }_{T,\infty}\right)T}\frac{1}{\left(t-s\right)^{1/\left(2\alpha\right)}}
			,\quad 0\le s <t\le T.
		\end{equation}
	\end{theorem}
	We are able to extend  the previous claim to functions $\phi\in C_b\left(\mathbb{R}^N\right)$ with an approximation procedure, effectively making Theorem \ref{zhang} a regularization--by--noise result. We need the next estimate, which derives from \cite[Eq. $(14)$]{2}:
	\begin{equation}\label{sub}
		\mathbb{E}\left[\frac{1}{L_t^{p}}\right]^{\frac{1}{p}}\le c\,t^{-\frac{1}{\alpha}},\quad t>0, \text{ for some $c=c\left({\alpha,p}\right)>0$, for every $p>0$.}
	\end{equation}
	\begin{corollary}\label{cor}
		Let $\alpha\in\left(\frac{1}{2},1\right)$ and $B_0\in C_b^{0,1}\left(\left[0,T\right]\times\mathbb{R}^N ;\mathbb{R}^N\right)$. Then, for every $\phi\in C_b\left(\mathbb{R}^N\right)$ and $0\le s <t\le T$, the function $P_{s,t}\phi$ is differentiable at $x\in\mathbb{R}^N$ in every direction $h\in\mathbb{R}^N$, and the expression in \eqref{bel} holds true. 
	\end{corollary}
	\begin{proof}
		Fix $x,h\in\mathbb{R}^N, 0\le s < t\le T$,  and  $\phi\in C_b\left(\mathbb{R}^N\right)$. Since $C_b^{\infty}\left(\mathbb{R}^N\right)$ is dense in $C_b\left(\mathbb{R}^N\right)$, we can take a sequence $\left(\phi_n\right)_n\subset C_b^{\infty}\left(\mathbb{R}^N\right)$ such that $\norm{\phi_n-\phi}_\infty\to 0$ as $n\to\infty$. Denote by $g_n\left(u\right)=P_{s,t}\phi_n\left(x+uh\right),\,u\in\mathbb{R}$; by dominated convergence, for every $u\in\mathbb{R}$,
		\[
		g_n\left(u\right)\to
		P_{s,t}\phi\left(x+uh\right)
		\eqqcolon g\left(u\right),\quad \text{as } n\to\infty.
		\]
		Now we invoke \eqref{bel}  to write 
		\begin{multline*}
			g_n'\left({u}\right)=\lim_{v\to 0}\frac{\mathbb{E}\left[\phi_n\left(X_t^{s,x+uh+vh}\right)\right]-\mathbb{E}\left[\phi_n\left(X_t^{s,x+uh}\right)\right]}{v}
			=\left\langle\nabla^\top P_{s,t}\phi_n\left(x+uh\right),h\right\rangle\\
			=
			\mathbb{E}\left[\frac{1}{L_t-L_s}\phi_n\left(X_t^{s,x+uh}\right)\int_{s}^{t}\left\langle\left(\sqrt{Q}\right)^{-1}D_hX^{s,x+{u}h}_r , dW_{L_r}\right\rangle\right],\quad {u}\in\mathbb{R}.
		\end{multline*}
		Since $\alpha\in\left(\frac{1}{2},1\right),$ an application of \cite[Theorem $3.2$]{Z}, \eqref{sub}, H\"older's inequality with $p\in\left(1,2\alpha\right)$ and Lemma~\ref{hoc} (see \eqref{boundala}) let us  compute
		\begin{multline}\label{pr_i}
			\sup_{u\in\mathbb{R}}	\mathbb{E}\left[\left|\frac{1}{L_t-L_s}\left(\phi_n\left(X_t^{s,x+{u}h}\right)-\phi\left(X_t^{s,x+{u}h}\right)\right)
			\int_{s}^{t}\left\langle\left(\sqrt{Q}\right)^{-1}D_hX^{s,x+{u}h}_r , dW_{L_r}\right\rangle\right|\right]\\\le \left|\left(\sqrt{Q}\right)^{-1}\right|\left|h\right|
			\frac{c_{1}}{\left(t-s\right)^{1/\alpha}}\norm{\phi_n-\phi}_\infty\to0,\quad \text{as }n\to \infty,
		\end{multline}
		where $c_{1}=c_1\left(\alpha,p,A,B_0,T,N\right)>0$. It follows that 
		\[
		g_n'\to \mathbb{E}\left[\dfrac{1}{L_t-L_s}\phi\left(X_t^{s,x+\left(\cdot\right) h}\right)\int_{s}^{t}\left\langle\left(\sqrt{Q}\right)^{-1}D_hX^{s,x+\left(\cdot\right)h}_r, dW_{L_r}\right\rangle\right], \quad \text{uniformly in }\mathbb{R}.
		\]
		This suffices to obtain the desired result, hence the proof is complete.
	\end{proof}
	Note that for every $\phi\in C_b\left(\mathbb{R}^N\right)$ the expression on the right--hand side of \eqref{bel} is continuous in $x$ for every $h\in\mathbb{R}^N$. Indeed, let us fix $x\in\mathbb{R}^N$ and consider $\left(x_n\right)_n\subset\mathbb{R}^N$ such that $x_n\to x$ as $n\to \infty.$ Then, using the same techniques as in the previous proof (cfr. \eqref{pr_i}), together with Lemma \ref{hoc} and a dominated convergence argument, we get (for some $p,q>1$ determined by a generalized Holder's inequality, and $c_{}=c\left(\alpha,p,q,A,B_0,T,N\right)>0$)
	\begin{align*}
		\begin{split}
			\left|\mathbb{E}\left[\frac{1}{L_t-L_s}\left(\phi\left({X}_t^{s,x_n}\right)\int_{s}^{t}\left\langle\left(\sqrt{Q}\right)^{-1}D_h{X}^{s,x_n}_r, dW_{L_r}\right\rangle-\phi\left({X}_t^{s,x}\right)\int_{s}^{t}\left\langle\left(\sqrt{Q}\right)^{-1}D_h{X}^{s,x}_r, dW_{L_r}\right\rangle\right)\right]\right|
			\\\le 
			\norm{\phi}_\infty\mathbb{E}\left[\frac{1}{L_t-L_s}\left|\int_{s}^{t}\left\langle\left(\sqrt{Q}\right)^{-1}\left(D_h{X}^{s,x_n}_r-D_h{X}^{s,x}_r\right), dW_{L_r}\right\rangle\right|\right]\\
			+
			\mathbb{E}\left[\frac{1}{L_t-L_s}\left|\int_{s}^{t}\left\langle\left(\sqrt{Q}\right)^{-1}D_h{X}^{s,x}_r, dW_{L_r}\right\rangle\right|\left|\phi\left({X}^{s,x_n}_t\right)-\phi\left({X}^{s,x}_t\right)\right|\right] 
			\le
			\frac{c_{}}{\left(t-s\right)^{1/\alpha}}\left|\left(\sqrt{Q}\right)^{-1}\right|\\
			\times\!\left[\norm{\phi}_\infty\! \left(\int_{s}^{t}\!\mathbb{E}\left[\left|D_h{X}^{s,x_n}_r-D_h{X}^{s,x}_r\right|^{2\alpha}\right]\!dr\right)^{\frac{1}{2\alpha}}+ \left|h\right|\mathbb{E}\left[\left|\phi\left({X}^{s,x_n}_t\right)-\phi\left({X}^{s,x}_t\right)\right|^q\right]^{\frac{1}{q}}\right]\underset{n\to\infty}{\longrightarrow} 0.
		\end{split}
	\end{align*}
	Therefore, $P_{s,t}\phi\in C_b^1\left(\mathbb{R}^N\right)$ for every $\phi\in C_b\left(\mathbb{R}^N\right).$ At this point, the next result is a straightforward consequence of the Chapman--Kolmogorov equations, the mean value theorem  and \cite[Lemma $7.1.5$]{DPZ2}.
	\begin{lemma}\label{bou}
		Let $\alpha\in\left(\frac{1}{2},1\right)$ and $B_0\in C_b^{0,1}\left(\left[0,T\right]\times\mathbb{R}^N ;\mathbb{R}^N\right)$. Then, for every $\phi\in \mathcal{B}_b\left(\mathbb{R}^N\right)$ and $0\le s<t\le T$, one has $P_{s,t}\phi\in C_b^1\left(\mathbb{R}^N\right)$, and the gradient estimate in \eqref{est_z} holds true.
	\end{lemma}

	Finally we are  in position to prove Theorem \ref{connection}.
	\begin{myproof}{Theorem}{\ref{connection}}
		Fix $\alpha\in\left(\frac12,1\right),\,0<t\le T,\,f\in C\left(\left[0,T\right];\mathbb{R}^N\right),\,B_0\in C^{0,3}_b\left(\left[0,T\right]\times \mathbb{R}^N;\mathbb{R}^N\right)$ and define $B=B_0-f$; we first consider $\phi\in C_b^3\left(\mathbb{R}^N\right).$ Recalling \eqref{OU}, we introduce the family of integro--differential operators $\big(\widetilde{A}\left(s\right)\big)_{0\le s\le T}$, defined for every $\psi\in C^2_b\left(\mathbb{R}^N\right)$ by
		\begin{equation*}
			\widetilde{A}\left(s\right)\psi\left(x\right)=\left\langle Ax+f\left(s\right), \nabla^\top \psi \left(x\right)\right\rangle + \int_{\mathbb{R}^N}\left[\psi\left(x+\sqrt{Q}z\right)-\psi\left(x\right)-1_D\left(z\right)\nabla \psi\left(x\right)\sqrt{Q}z\right]\nu\left(dz\right),
		\end{equation*}
		where $x\in\mathbb{R}^N$. Let us take $0\le s < t,\,x\in\mathbb{R}^N,$ and observe that by the definition in \eqref{operator} and Corollary \ref{higher}  there exists a constant $C>0$ such that, for every $r_1,\,r_2\in\left[s,t\right]$,
		\begin{multline}\label{functional}
			\sup_{u\in\left[s,t\right]}\left|A\left(u\right)P_{u,t}\phi\left(Z^{s,x}_{r_2}\right)-A\left(u\right)P_{u,t}\phi\left(Z^{s,x}_{r_1}\right)\right|
			\\\le  C
			\left[\left(\left|A\right|\left(1+\left|Z^{s,x}_{r_1}\right|\right)+\norm{B_0}_{1,T}\right)
			+\int_{\mathbb{R}^N}\left(1_D\left(z\right)\left|\sqrt{Q}z\right|^2+1_{D^c}\left(z\right)\right)\nu\left(dz\right)\right]\left|Z^{s,x}_{r_2}-Z^{s,x}_{r_1}\right|.
		\end{multline}
		We study the mapping $\left[s,t\right]\ni r\mapsto R_{s,r}\left(P_{r,t}\phi\right)\left(x\right)$: using \eqref{ste1} and \eqref{functional}, it is easy to argue that it is continuous in its domain by Theorem \ref*{kfb} \emph{\ref{back_k}} coupled with Vitali's and dominated convergence theorems. It is also differentiable, with
		\begin{multline}\label{finally}
			\partial_rR_{s,r}\left(P_{r,t}\phi\right)\left(x\right)=R_{s,r}\left(\widetilde{A}\left(r\right)P_{r,t}\phi\right)\left(x\right)
			-R_{s,r}\left(A\left(r\right)P_{r,t}\phi\right)\left(x\right)\\
			=-
			R_{s,r}\left(\left\langle B\left(r,\cdot\right) ,\nabla^\top P_{r,t}\phi \right\rangle \right)\left(x\right)
			,\quad r\in\left[s,t\right].
		\end{multline}
		Indeed, take  $r\in\left[s,t\right]$ and a generic sequence $\left(r_n\right)_n\subset\left[s,t\right]\setminus\left\{r\right\}$ such that $r_n\to r$ as $n\to \infty$; then 
		\begin{multline*}
			\frac{R_{s,r_n}\left(P_{r_n,t}\phi\right)\left(x\right)-R_{s,r}\left(P_{r,t}\phi\right)\left(x\right)}{r_n-r}=
			R_{s,r_n}\left(\frac{P_{r_n,t}\phi-P_{r,t}\phi}{r_n-r}\right)\left(x\right)+
			\mathbb{E}\left[\frac{P_{r,t}\phi\left(Z^{s,x}_{r_n}\right)-P_{r,t}\phi\left(Z^{s,x}_{r}\right)}{r_n-r}\right]
			\\
			\eqqcolon \mathbf{\upperRomannumeral{1}}_n+\mathbf{\upperRomannumeral{2}}_n.
		\end{multline*}
		We immediately notice that $\mathbf{\upperRomannumeral{2}}_n\to R_{s,r}\left(\widetilde{A}\left(r\right)P_{r,t}\phi\right)\left(x\right)$ as $n\to\infty$ by Theorem \ref*{kfb} \emph{\ref{for_k}} and Corollary \ref{higher}. As for $\mathbf{\upperRomannumeral{1}}_n$, we split it again as follows:
		\[
		\mathbf{\upperRomannumeral{1}}_n=R_{s,r}\left(\frac{P_{r_n,t}\phi-P_{r,t}\phi}{r_n-r}\right)\left(x\right)+
		\mathbb{E}\left[\frac{P_{r_n,t}\phi-P_{r,t}\phi}{r_n-r}\left(Z^{s,x}_{r_n}\right)-\frac{P_{r_n,t}\phi-P_{r,t}\phi}{r_n-r}\left(Z^{s,x}_r\right)\right]\eqqcolon 
		\mathbf{\upperRomannumeral{3}}_n+\mathbf{\upperRomannumeral{4}}_n.
		\]
		By a dominated convergence argument based on \eqref{salvare}, \eqref{ste1}, Corollary \ref{higher} and Theorem \ref*{kfb} \emph{\ref{back_k}} we have $\mathbf{\upperRomannumeral{3}}_n\to-R_{s,r}\left(A\left(r\right)P_{r,t}\phi\right)\left(x\right)$ as $n\to\infty$. Finally we focus on $\mathbf{\upperRomannumeral{4}}_n$, estimating by \eqref{ste1}
		\begin{equation*}
			\left|\mathbf{\upperRomannumeral{4}}_n\right|\le \mathbb{E}\left[\sup_{u\in\left[s,t\right]}\left|A\left(u\right)P_{u,t}\phi\left(Z^{s,x}_{r_n}\right)-A\left(u\right)P_{u,t}\phi\left(Z^{s,x}_{r}\right)\right|\right].
		\end{equation*}
		Notice that the random variables inside the expected value in the previous inequality  converge to $0$ in probability as $n\to \infty$ by \eqref{functional}. Such a convergence is true also in the $L^1-$sense, thanks to the estimates in  \eqref{salvare} and Vitali's convergence theorem. Thus, $\mathbf{\upperRomannumeral{4}}_n\to 0$ as $n\to\infty$, fact which completely shows \eqref{finally}. Observe that $\partial_rR_{s,r}\left(P_{r,t}\phi\right)\left(x\right)$ is continuous in $\left[s,t\right]$ by Vitali's and dominated convergence theorems, the mean value theorem, Corollary \ref{higher} and the continuity of the mapping $r\mapsto \nabla P_{r,t}\phi\left(x\right)$ in $\left[s,t\right]$ (see \eqref{ste1} and the subsequent sentence). Therefore we can integrate it with respect to $r$ on the interval $\left[s,t\right]$ and infer that
		\begin{equation}\label{pr_1}
			P_{s,t}\phi\left(x\right)=R_{s,t}\phi\left(x\right)+\int_{s}^{t}R_{s,r}\left(\left\langle B\left(r,\cdot\right),\nabla^\top P_{r,t}\phi\right\rangle\right)\left(x\right)dr,
		\end{equation}
		which coincides with \eqref{Kolm_mild}.

		Next, we take $\phi\in C_b\left(\mathbb{R}^N\right)$ and consider a sequence $\left(\phi_n\right)_n\subset C^3_b\left(\mathbb{R}^N\right)$ such that $\norm{\phi_n-\phi}_\infty\to 0$ as $n\to \infty$. Since by \eqref{est_z} and Lemma \ref{bou} (for some constant ${C}_\alpha>0$)
		\begin{multline*}
			\left|\int_{s}^{t}R_{s,r}\left(\left\langle B\left(r,\cdot\right), \nabla^\top P_{r,t}\left(\phi_n-\phi\right)\right\rangle\right)\left(x\right)dr\right|\\\le 
			{C}_\alpha
			\norm{B}_{0,T}\norm{\phi_n-\phi}_\infty
			\left|\left(\sqrt{Q}\right)^{-1}\right|e^{\left(\left|A\right|+\norm{DB_0}_{T,\infty}\right)T}\left(\int_{s}^{t}\frac{dr}{\left(t-r\right)^{1/\left(2\alpha\right)}}\right)
			\underset{n\to\infty}{\longrightarrow}0,
		\end{multline*}
		by dominated convergence it is immediate to get the validity of \eqref{pr_1} for $\phi$, as well.
		
		Finally, we tackle the case $\phi\in\mathcal{B}_b\left(\mathbb{R}^N\right)$. We consider $\phi$ to be the indicator function of an open set  to begin with. Then, by Urysohn’s lemma there exists a sequence $\left(\phi_n\right)_n\subset C_b\left(\mathbb{R}^N\right)$ such that $0\le \phi_n\le \phi$ and $\phi_n\to\phi$ pointwise as $n\to\infty$. By construction and dominated convergence we have
		\begin{equation}\label{pr_2}
			\lim_{n\to \infty}\left(P_{s,t}\phi_n\left(x\right)-R_{s,t}\phi_n\left(x\right)\right)= P_{s,t}\phi\left(x\right)-R_{s,t}\phi\left(x\right).
		\end{equation} 
		Now we focus on the integral term in \eqref{pr_1}. Let us fix $y,h\in\mathbb{R}^N,\,r\in\left(s,t\right)$ and $u\in\left(r,t\right)$. Then, exploiting the Chapman--Kolmogorov equations and \eqref{bel}, we write ($n\in\mathbb{N}$)
		\begin{multline}\label{pr_3}
			\left\langle \nabla^\top P_{r,t}\phi_n\left(y\right),h\right\rangle
			=
			\left\langle \nabla^\top\left(P_{r,u}\left(P_{u,t}\phi_n\right)\right)\left(y\right),h\right\rangle\\=
			\mathbb{E}\left[\frac{1}{L_u-L_r}P_{u,t}\phi_n\left(X^{r,y}_u\right)\int_{r}^{u}\left\langle\left(\sqrt{Q}\right)^{-1} D_hX_v^{r,y}, dW_{L_v}\right\rangle \right].
		\end{multline}
		Since, with the same argument as in \eqref{pr_2}, $P_{u,t}\phi_n\to P_{u,t}\phi$ pointwise in $\mathbb{R}^N$ as $n\to\infty,$ and (see, e.g., \eqref{pr_i}) 
		\begin{equation*}
			\sup_{n\in\mathbb{N}}\left|\frac{P_{u,t}\phi_n\left(X^{r,y}_u\right)}{L_u-L_r}\int_{r}^{u}\left\langle\left(\sqrt{Q}\right)^{-1} D_hX_v^{r,y}, dW_{L_v}\right\rangle\right|\le \frac{1}{L_u-L_r}\left|\int_{r}^{u}\left\langle\left(\sqrt{Q}\right)^{-1}D_hX_v^{r,y}, dW_{L_v}\right\rangle\right|\in L^1\left(\mathbb{P}\right),
		\end{equation*}
		we can pass to the limit in \eqref{pr_3} to obtain, by dominated convergence,
		\begin{multline*}
			\lim_{n\to \infty}\left\langle \nabla^\top P_{r,t}\phi_n\left(y\right),h\right\rangle=
			\mathbb{E}\left[\frac{1}{L_u-L_r}P_{u,t}\phi\left(X^{r,y}_u\right)\int_{r}^{u}\left\langle\left(\sqrt{Q}\right)^{-1} D_hX_v^{r,y}, dW_{L_v}\right\rangle \right]\\
			=\left\langle \nabla^\top\left(P_{r,u}\left(P_{u,t}\phi\right)\right)\left(y\right),h\right\rangle
			=	\left\langle \nabla^\top  P_{r,t}\phi\left(y\right),h\right\rangle.
		\end{multline*}
		Observe that the second--to--last equality in the previous equation is due to \eqref{bel} and  Lemma \ref{bou}. As a consequence, for every $r\in\left(s,t\right)$ we infer that
		\begin{equation*}
			\lim_{n\to \infty}R_{s,r}\left(\left\langle B\left(r,\cdot\right),\nabla^\top P_{r,t}\phi_n\right\rangle\right)\left(x\right)=
			R_{s,r}\left(\left\langle B\left(r,\cdot\right),\nabla^\top P_{r,t}\phi\right\rangle\right)\left(x\right),
		\end{equation*}
		where we use once again the dominated convergence theorem, thanks to the next bound that we get using \eqref{est_z} and Lemma \ref{bou}:
		\[
		\norm{\left\langle B\left(r,\cdot\right),\nabla^\top P_{r,t}\phi_n\right\rangle}_\infty
		\le {C}_\alpha 
		\norm{B}_{0,T}\left|\left(\sqrt{Q}\right)^{-1}\right|e^{\left(\left|A\right|+\norm{DB_0}_{T,\infty}\right)T}\frac{1}{\left(t-r\right)^{1/\left(2\alpha\right)}}.
		\]
		Moreover, this inequality also allows to pass the limit under the integral sign, so that we end up with
		\begin{equation}\label{pr_f}
			\lim_{n\to \infty}\int_{s}^{t}R_{s,r}\left(\left\langle B\left(r,\cdot\right),\nabla^\top P_{r,t}\phi_n\right\rangle\right)\left(x\right)dr=\int_{s}^{t}R_{s,r}\left(\left\langle B\left(r,\cdot\right),\nabla^\top P_{r,t}\phi\right\rangle\right)\left(x\right)dr.
		\end{equation}
		Combining \eqref{pr_2}-\eqref{pr_f} we conclude that \eqref{pr_1} holds true for $\phi$, i.e., for every indicator function of an open set. \\
		Note that the passages of the previous step do not require the continuity of the approximating functions $\left(\phi_n\right)_n$, as long as they are equibounded, satisfy \eqref{pr_1} and converge pointwise to $\phi$. Therefore, we can state that \eqref{pr_1} holds true for every $\phi\in\mathcal{B}_b\left(\mathbb{R}^N\right)$ by the functional monotone class theorem (see, e.g., \cite[Theorem $2.12.9$]{boga}).
		
		We notice that, from \eqref{pr_1}, the continuity of $P_{\cdot, t}\phi\left(x\right),\,x\in\mathbb{R}^N$, in the interval $\left[0,t\right)$ can be argued by dominated convergence (see \eqref{continuity} below for an analogous computation). Furthermore, the measurability of $P_{s,t}\phi\left(x\right)$ with respect to $\left(s,x\right)$ is a consequence of the measurability of the stochastic flow $X^{s,x}_t\left(\omega\right)$ and Tonelli's theorem.
		These facts, together with Lemma \ref{bou}  and the gradient estimate in \eqref{est_z}, entail that $P_{t-\diamond,t}\phi\left(\cdot\right)\in\Lambda^\gamma_1\left(0,t\right],\,\gamma=1/\left(2\alpha\right).$ Recalling Theorem \ref{well_pos} the proof is complete.
	\end{myproof}
	\begin{rem}\label{rem1}
		Suppose that the requirements of Theorem \ref{connection} are satisfied. Given $0\le s<t\le T$ and $\phi\in \mathcal{B}_b\left(\mathbb{R}^N\right)$, we consider $r\in\left(s,t\right)$ and call $\widetilde{\phi}=P_{r,t}\phi$. By Theorem~\ref{connection} and the Chapman--Kolmogorov equations,
		\[
		P_{s,t}\phi\left(x\right)=P_{s,r}\widetilde{\phi}\left(x\right)
		=
		u_s^{\widetilde{\phi}}\left(r,x\right),\quad x\in\mathbb{R}^N,
		\]  
		where $u_s^{\widetilde{\phi}}\left(r,x\right)$ is the unique solution of \eqref{Kolm_mild} such that $u^{\widetilde{\phi}}_{r-\diamond}\left(r,\cdot\right)\in\Lambda^\gamma_1\left(0,r\right],\,\gamma=1/\left(2\alpha\right)$. 
		Observing that $\widetilde{\phi}\in C_b^1\left(\mathbb{R}^N\right)$ by Lemma \ref{bou}, we invoke Corollary \ref{cor2} to say that $P_{s,t}\phi\in C^2_b\left(\mathbb{R}^N\right).$ An iteration of this argument shows that $P_{s,t}\phi\in C^4_b\left(\mathbb{R}^N\right)$. In particular, the Kolmogorov backward equation \eqref{kol_back} holds true in the interval $\left[0,t\right)$ for every $\phi\in\mathcal{B}_b\left(\mathbb{R}^N\right)$.
	\end{rem}
	\section{The iteration scheme}\label{ite_sc}
	Let $\alpha\in\left(\frac{1}{2},1\right),\,t\in\left(0,T\right],\,u_0\in\mathcal{B}_b\left(\mathbb{R}^N\right)$ and consider $B_0\in C_b^{0,3}\left(\left[0,T\right]\times \mathbb{R}^N;\mathbb{R}^N\right),\,f\in C\left(\left[0,T\right];\mathbb{R}^N\right)$, so that Theorem \ref{connection} holds true. The proof of Theorem \ref{well_pos} (see, in particular, \eqref{Gamma}-\eqref{u}) suggests to approximate the unique solution $u^{u_0}_s\left(t,x\right)(=P_{s,t}u_0\left(x\right))$ of \eqref{Kolm_mild} such that $u^{u_0}_{t-\diamond}\left(t,\cdot\right)\in \Lambda^\gamma_1\left(0,t\right],\,\gamma=1/\left(2\alpha\right),$ with the iterates
	\begin{equation*}
		\begin{cases}
			u^{n+1}_s\left(t,x\right)=R_{s,t}u_0\left(x\right)+\int_{s}^{t}R_{s,r}\left(\left\langle B\left(r,\cdot\right),\nabla^\top u_r^{n}\left(t,\cdot\right)\right\rangle\right)\left(x\right)dr\\
			u^0_s\left(t,x\right)=R_{s,t}u_0\left(x\right)
		\end{cases},\quad x\in\mathbb{R}^N,\,s\in\left[0,t\right],\,n\in\mathbb{N}\cup \left\{0\right\}.
	\end{equation*}
Here we recall that $B=B_0-f.$  
	If we define $v^0_s\left(t,x\right)= u^0_s\left(t,x\right)$ and $v^{n+1}_s\left(t,x\right)= u^{n+1}_s\left(t,x\right)-u^{n}_s\left(t,x\right),\, n\in\mathbb{N}\cup \left\{0\right\}$, then these new functions satisfy the iteration scheme
	\begin{equation}\label{it_scheme}
		\begin{cases}
			v_s^{n+1}\left(t,x\right)=\int_{s}^{t}R_{s,u}k_{u,t}^{n}\left(x\right)du\\
			k^{n}_{u,t}\left(x\right)=\left\langle B\left(u,x\right), \nabla^\top v_u^{n}\left(t,x\right)\right\rangle\\
			v_s^{0}\left(t,x\right)=R_{s,t}u_0\left(x\right)
		\end{cases},
	\quad x\in\mathbb{R}^N,\,\,s\in\left[0,t\right],\,u\in\left[0,t\right),\,n\in\mathbb{N}\cup\left\{0\right\}.
	\end{equation}
	In the Brownian case, \eqref{it_scheme} has been investigated in \cite{LRF2}. In order to study the convergence of $\sum_{n=0}^{\infty}v^n_s\left(t,x\right)$ to $u^{u_0}_s\left(t,x\right)$ (in a sense that will be clarified later on), we need the next, preliminary result.
	\begin{lemma}\label{cont_k}
		Let $\alpha\in\left(\frac{1}{2},1\right),\,t\in \left(0,T\right],\,n\in\mathbb{N}\cup \left\{0\right\}$ and denote by $\gamma=1/\left(2\alpha\right)$. Then $k^n_{u,t}\in C_b\left(\mathbb{R}^N\right)$ and   $v^n_s\left(t,\cdot\right)\in C_b^1\left(\mathbb{R}^N\right)$ for every $u,\,s\in\left[0,t\right)$.
		
		Moreover, there exists a constant $C=C\left(\alpha,A,Q\right)>0$  such that, for every $n\in\mathbb{N}$ and $s\in\left[0,t\right)$,
		\begin{equation}\label{v_n_est}
			\norm{v^n_{s}\left(t,\cdot\right)}_{\infty}\le C^{n}\norm{B}_{0,T}^{n}\norm{u_0}_{\infty}
			\int_{0}^{t-s}ds_n\int_{0}^{s_n}ds_{n-1}\dots\int_{0}^{s_{2}}ds_1
			\prod_{i=1}^{n}\frac{1}{\left(s_{i+1}-s_{i}\right)^{\gamma}},
		\end{equation}
		and
		\begin{equation}\label{v^n_gr_est}
			\norm{\nabla^\top v^n_{s}\left(t,\cdot\right)}_{\infty}\le C^{n+1}\norm{B}_{0,T}^{n}\norm{u_0}_{\infty}
			\int_{0}^{t-s}ds_n\int_{0}^{s_n}ds_{n-1}\dots\int_{0}^{s_{2}}ds_1
				\prod_{i=0}^{n}\frac{1}{\left(s_{i+1}-s_{i}\right)^{\gamma}},
		\end{equation}
		where $s_0=0$ and $s_{n+1}=t-s$.
	\end{lemma}
	\noindent We notice that the constant $C$ in \eqref{v_n_est}-\eqref{v^n_gr_est} is the same as the one appearing in the gradient estimate \eqref{est12}.
	\begin{proof}
		We proceed  by induction to prove that, for every $u,s\in\left[0,t\right)$ and $n\in\mathbb{N}\cup\left\{0\right\}$, one has  $v^{n}_{s}\left(t,\cdot\right)\in C_b^1\left(\mathbb{R}^N\right), \,k^{n}_{u,t}\in C_b\left(\mathbb{R}^N\right)$ and 
		\begin{equation}\label{k^n_gr_est}
			\norm{k^n_{u,t}}_{\infty}\le C^{n+1}\norm{B}_{0,T}^{n+1}\norm{u_0}_{\infty}
			\int_{u}^{t}ds_1\int_{s_1}^{t}ds_{2}\dots\int_{s_{n-1}}^{t}ds_n
			\prod_{i=0}^{n}\frac{1}{\left(s_{i+1}-s_{i}\right)^{\gamma}},
		\end{equation}
		where $C=C\left(\alpha,A,Q\right)>0$ is the same constant as in \eqref{est12}. In \eqref{k^n_gr_est}, $s_0=u$ and $s_{n+1}=t.$ The estimates in \eqref{v_n_est}-\eqref{v^n_gr_est} are an immediate consequence of \eqref{k^n_gr_est} upon shifting the domain of integration and applying Tonelli's theorem.
		\\For $n=0$, the smoothing effect of the time--dependent Markov semigroup $R$ guarantees that $v^0_s\left(t,\cdot\right)\in C_b^1\left(\mathbb{R}^N\right)$, which combined with the continuity of $B$ yields $k^0_{u,t}\in C_b\left(\mathbb{R}^N\right),$ with 
		\begin{equation}\label{l3_1}
			\norm{k^0_{u,t}}_\infty\le C\norm{B}_{0,T}\norm{u_0}_\infty\frac{1}{\left(t-u\right)^\gamma}.
		\end{equation}
		To fix the ideas, consider the case $n=1$. Since $k^0_{u,t}\in C_b\left(\mathbb{R}^N\right)$ for every $0\le u<t$, the dominated convergence theorem, \eqref{it_scheme} and \eqref{l3_1} imply that $v^1_s\left(t,\cdot\right)\in C^1_b\left(\mathbb{R}^N\right),$ with $\nabla v_s^{1}\left(t,x\right)=\int_{s}^{t}\nabla R_{s,u}k_{u,t}^{0}\left(x\right)du,\,x\in\mathbb{R}^N$. Hence $k^1_{u,t}\in C_b\left(\mathbb{R}^N\right)$, and by \eqref{est12}-\eqref{l3_1} we get
		\[
		\norm{k^1_{u,t}}_\infty\le C^2\norm{B}_{0,T}^2\norm{u_0}_\infty
		\int_{u}^{t}ds_1\frac{1}{\left(s_1-u\right)^\gamma\left(t-s_1\right)^\gamma}.
		\]
		Suppose now that our statement holds true at step $n\in\mathbb{N}$. Then by the same argument as before and \eqref{k^n_gr_est} $v^{n+1}_s\left(t,\cdot\right)\in C_b^1\left(\mathbb{R}^N\right)$, with $\nabla v_s^{n+1}\left(t,x\right)=\int_{s}^{t}\nabla R_{s,u}k_{u,t}^{n}\left(x\right)du$. Therefore $k^{n+1}_{u,t}\in C_b\left(\mathbb{R}^N\right)$, with
		\begin{multline*}
			\norm{k^{n+1}_{u,t}}_{\infty}\le C \norm{B}_{0,T} \int_{u}^{t}ds_{1}\frac{1}{\left(s_{1}-u\right)^{\gamma}}\norm{k^n_{s_{1},t}}_{\infty}\\
			\le
			C^{n+2}\norm{B}^{n+2}_{0,T}\norm{u_0}_{\infty}\int_{u}^{t}ds_{1}\int_{s_{1}}^{t}ds_2\dots\int_{s_n}^{t}ds_{n+1}\prod_{i=0}^{n+1}\frac{1}{\left(s_{i+1}-s_{i}\right)^{\gamma}},
		\end{multline*} 
		where in the last inequality we apply the inductive hypothesis and consider $s_0=u,\,s_{n+2}=t$.
		Thus, the claim is completely proved.
	\end{proof}
	Another important property of the functions $v^n_\cdot\left(t,x\right),\,x\in\mathbb{R}^N$, is the continuity in the interval $\left[0,t\right)$. In the case $n=0$, this follows from the property of $R$ discussed in Section \ref{pre}\,; for a generic $n\in\mathbb{N},$ it can be argued by \eqref{k^n_gr_est} and dominated convergence writing
	\begin{equation}\label{continuity}
		v^n_s\left(t,x\right)=\int_{0}^{t}1_{\left\{u>s\right\}}R_{s,u}k^{n-1}_{u,t}\left(x\right)du.
	\end{equation}
	Thanks to the estimates in \eqref{v_n_est}-\eqref{v^n_gr_est},  the convergence of the iteration scheme \eqref{it_scheme} is proved in the same way as in the Brownian case with no time--shift, see \cite[Section $2.4$]{LRF1}. Overall, the next result is true.
	\begin{theorem}
		For every $\alpha\in\left(\frac{1}{2},1\right)$ and $0<t\le T$, the series $\sum_{n=0}^\infty v^n_s\left(t,x\right)$ converges uniformly in $\left[0,t\right]\times \mathbb{R}^N$, and the series $\sum_{n=0}^\infty \nabla^\top v^n_s\left(t,x\right)$ converges uniformly in $\left[0,t_0\right]\times \mathbb{R}^N$, for every $t_0\in\left(0,t\right)$. In particular,
		\begin{equation*}
			\sum_{n=0}^\infty v^n_s\left(t,x\right)=
			u^{u_0}_s\left(t,x\right),\quad s\in\left[0,t\right],\,x\in\mathbb{R}^N,
		\end{equation*}
		where $u^{u_0}_s\left(t,x\right)$ is the unique solution  of \eqref{Kolm_mild} such that $u^{u_0}_{t-\diamond}\left(t,\cdot\right)\in \Lambda^\gamma_1\left(0,t\right],\,\gamma=1/\left(2\alpha\right)$.
	\end{theorem}
	\section{The first term of the iteration scheme}\label{first}
	Let $\alpha\in\left(\frac{1}{2},1\right)$. The goal of this section is to study the first term $v_s^1\left(t,x\right)=\int_{s}^{t}R_{s,u}k^0_{u,t}\left(x\right)du$ of \eqref{it_scheme} for every $0\le s <t\le T$ and $x\in\mathbb{R}^N$. In particular, starting from 
	\begin{equation}\label{k^0_def}
		k^0_{u,t}\left(y\right)=\left\langle B\left(u,y\right),\nabla^\top R_{u,t}u_0\left(y\right)\right\rangle,\quad y\in \mathbb{R}^N,\,u\in\left(s,t\right),
	\end{equation}
	we want to find an alternative, explicit expression (see Lemma \ref{l13}) for 
	\begin{equation}\label{1goal}
		R_{s,u}k^0_{u,t}\left(x\right)=\mathbb{E}\left[k^0_{u,t}\left(Z_{u}^{s,x}\right)\right].
	\end{equation} 
In order to do this, we propose an approach which at first analyzes a deterministic time--shift, and then allows to recover the subordinated Brownian motion case by conditioning with respect to $\mathcal{F}^L$. The results of this part represent the base case for the induction argument that we will develop to compute the general term $v^{n+1}_s \left(t, x\right), \,n\ge 1$ (see Section \ref{sec_3}).
	
	\subsection{Deterministic time--shift}\label{sub1.1}	
	Denote by $\mathbb{S}$ the set of real--valued, strictly increasing càdlàg functions defined on $\mathbb{R}_+$ and  starting at\,$0$. Take $\ell\in\mathbb{S}$ and note that $W_\ell=\left(W_{\ell_t}\right)_{t\ge0 }$ is a càdlàg martingale with respect to the filtration $\left(\mathcal{F}^W_{\ell_t}\right)_{t\ge0}$, where $\left(\mathcal{F}^W_t\right)_{t\ge0}$ is the minimal augmented filtration generated by $W$. For every $x\in\mathbb{R}^N$ and $0\le s <T$, the  OU process $\left(Z_t^\ell\left(s,x\right)\right)_{t\in\left[s,T\right]}$ is the unique, càdlàg solution of the linear SDE 
	\[
	dZ^{\ell}_t\left(s,x\right)=\left(AZ^{\ell}_t\left(s,x\right)+f\left(t\right)\right)dt+\sqrt{Q}\,dW_{\ell_t},
	\quad Z^{\ell}_s\left(s,x\right)=x.
	\]
	It can be expressed with a variation of constants formula as follows: 
	\begin{equation*}
		Z^\ell_t\left(s,x\right)=e^{\left(t-s\right)A}x+\int_{s}^{t}e^{\left(t-r\right)A}f\left(r\right)dr+\int_{s}^{t}e^{\left(t-r\right)A}\sqrt{Q}\,dW_{\ell_r},\quad t\in\left[s,T\right].
	\end{equation*}
	For every $0\le s<t\le T$, define 
	$I^\ell_{s,t}=\int_{s}^te^{2\left(t-r\right)A}Q\,d\ell_r\in\mathbb{R}^{N\times N}$. It is possible to  argue as in \cite[Equation~$\left(12\right)$]{AB} to deduce that 
	\[
	Z^\ell_{t}\left(s,x\right)\sim\mathcal{N}\left(e^{\left(t-s\right)A}x+F_{s,t},\,I^\ell_{s,t}\right).
	\] 
	Note that, for every $0\le s< u< t\le T$, 
	\[
	Z^\ell_t\left(s,x\right)=e^{\left(t-u\right)A}Z^\ell_{u}\left(s,x\right)
	+F^{}_{u,t}
	+\int_{u}^{t}e^{\left(t-r\right)A}\sqrt{Q}\,dW_{\ell_r},\quad \mathbb{P}-\text{a.s.,}
	\]
	therefore $\left(Z^\ell\left(s,x\right)\right)_{x\in\mathbb{R}^N}$ is a family of  $\left(\mathcal{F}_{\ell_t}\right)_{t\in\left[s,T\right]}-$Markov processes as $s$ varies in $\left[0,T\right)$. In particular, its transition probability kernels $\mu^{\ell}_{u,t}\colon \mathbb{R}^N\times \mathcal{B}\left(\mathbb{R}^N\right)\to\left[0,1\right]$ are  
	\begin{equation}\label{trans}
		\mu_{u,t}^{\ell}\left(y,\cdot\right)=\mathcal{N}\left(e^{\left(t-u\right)A}y+F^{}_{u,t},I^{\ell}_{u,t}\right),\quad y\in\mathbb{R}^N.
	\end{equation}
	In the sequel, we denote by $\phi_{u,t}^\ell\left(y,\cdot\right)$ the density of $\mu^\ell_{u,t}\left(y,\cdot\right).$
	Straightforward changes to  \cite[Theorem $4$]{AB} ensure that, for any $0\le s<t\le T,$ the function $\mathbb{E}\left[u_0\left(Z^\ell_t\left(s,\cdot\right)\right)\right]\in C_b^1\left(\mathbb{R}^N\right)$, with  derivative at any point $x\in\mathbb{R}^N$ in every direction $h\in\mathbb{R}^N$ given by 
	\begin{equation}\label{der_deter}
		\left\langle\nabla^\top\mathbb{E}\left[u_0\left(Z^\ell_t\left(s,x\right)\right)\right],h\right\rangle
		=
		\mathbb{E}\left[u_0\left(Z_t^\ell\left(s,x\right)\right)\left\langle \left(I_{s,t}^\ell\right)^{-1}e^{\left(t-s\right)A}h,Z^\ell_t\left(s,x\right)-e^{\left(t-s\right)A}x-F^{}_{s,t}\right\rangle\right].
	\end{equation}
	
	With all these preliminaries in mind, we fix $\ell^0\in\mathbb{S},\,0\le u<t\le T$ and define --by analogy with \eqref{k^0_def}-- the function 
	\begin{equation}\label{k0,1}
		k^{\ell^0}_{u,t}\left(y\right)=
		\left\langle B\left(u,y\right),\nabla^\top\mathbb{E}\left[u_0\left(Z^{\ell^0}_t\left(u,y\right)\right)\right]\right\rangle,\quad y\in\mathbb{R}^N.
	\end{equation}
	Note that $k^{\ell^0}_{u,t}\in C_b\left(\mathbb{R}^N\right)$ because $B\left(u,\cdot\right)$ is continuous and bounded, as well. The next claim provides us an analogue of \eqref{1goal} in this framework. 
	\begin{lemma}\label{det_time_1}
		Consider $0\le s<t\le T$. Then for every $x\in\mathbb{R}^N,\,u\in\left(s,t\right)$ and $\ell^0,\ell^1\in\mathbb{S}$, one has, $\mathbb{P}-$a.s.,
			\begin{align}
			\begin{split}\label{cond_k01}
				k^{\ell^0}_{u,t}\left(Z^{\ell^1}_u\left(s,x\right)\right)\\
				=\mathbb{E}\bigg[u_0\left(\left(I_{u,t}^{\ell^0}\right)^{\frac{1}{2}}\left(I_{u,t}^{\ell^1}\right)^{-\frac{1}{2}}\left(Z^{\ell^1}_{t}\left(s,x\right)-e^{\left(t-u\right)A}Z^{\ell^1}_{u}\left(s,x\right)-F^{}_{u,t}\right)+F^{}_{u,t}+e^{\left(t-u\right)A}Z^{\ell^1}_{u}\left(s,x\right)\right)\\
				\!\times\!\left\langle \left(I^{\ell^0}_{u,t}\right)^{-\frac{1}{2}}e^{\left(t-u\right)A}B\left(u,Z_{u}^{\ell^1}\left(s,x\right)\right),\left(I^{\ell^1}_{u,t}\right)^{\!\!-\frac{1}{2}}\!\!\left(Z^{\ell^1}_{t}\left(s,x\right)\!-e^{\left(t-u\right)A}Z^{\ell^1}_{u}\left(s,x\right)-F^{}_{u,t}\right)\right\rangle\!\Big|
				\sigma\left(Z^{\ell^1}_{u}\left(s,x\right)\right)\!\bigg].
			\end{split}
		\end{align}
	\end{lemma}
\begin{proof}
Fix $x\in\mathbb{R}^N,\,0\le s < u< t\le T$ and $\ell^0,\,\ell^1\in\mathbb{S}$; by \eqref{der_deter} we have 
\begin{multline}\label{det_1}
	k^{\ell^0}_{u,t}\left(Z^{\ell^1}_u\left(s,x\right)\right)
	=\restr{k^{\ell^0}_{u,t}\left(y\right)}{y=Z^{\ell^1}_u\left(s,x\right)}
	\\
	=\restr
	{\mathbb{E}\left[u_0\left(Z^{\ell^0}_t\left(u,y\right)\right)\left\langle \left(I^{\ell^0}_{u,t}\right)^{-1}e^{\left(t-u\right)A}B\left(u,y\right),Z^{\ell^0}_t\left(u,y\right)-e^{\left(t-u\right)A}y-F^{}_{u,t}\right\rangle\right]}{y=Z^{\ell^1}_u\left(s,x\right)}.
\end{multline}
Note that  $Z^{\ell^0}_t\left(u,y\right)\sim\mu^{\ell^0}_{u,t}\left(y,\cdot\right),\,y\in\mathbb{R}^N$; furthermore, direct computations show that, for every $y,\,\xi\in\mathbb{R}^N$, 
\begin{equation*}
	\phi^{\ell^1}_{u,t}\left(y,\xi\right)
	=\det\left(\left(I_{u,t}^{\ell^0}\right)^{\frac{1}{2}}\left(I_{u,t}^{\ell^1}\right)^{-\frac{1}{2}}\right)\phi_{u,t}^{\ell^0}\left(y,\left(I_{u,t}^{\ell^0}\right)^{\frac{1}{2}}\left(I_{u,t}^{\ell^1}\right)^{-\frac{1}{2}}\left(\xi-e^{\left(t-u\right)A}y-F^{}_{u,t}\right)+e^{\left(t-u\right)A}y+F^{}_{u,t}\right).
\end{equation*}
Going back to \eqref{det_1} we write, with the substitution $\xi=\left(I_{u,t}^{\ell^0}\right)^{\frac{1}{2}}\left(I_{u,t}^{\ell^1}\right)^{-\frac{1}{2}}\left(\xi'-e^{\left(t-u\right)A}y-F^{}_{u,t}\right)+e^{\left(t-u\right)A}y+F^{}_{u,t}$ suggested by the previous calculations,
\begin{align*}
	&\restr{k^{\ell^0}_{u,t}\left(y\right)}{y=Z^{\ell^1}_u\left(s,x\right)}
	=\restr{\int_{\mathbb{R}^N}u_0\left(\xi\right)\left\langle \left(I^{\ell^0}_{u,t}\right)^{-1}e^{\left(t-u\right)A}B\left(u,y\right),\xi-e^{\left(t-u\right)A}y-F^{}_{u,t}\right\rangle\phi_{u,t}^{\ell^0}\left(y,\xi\right)d\xi\,}{y=Z^{\ell^1}_u\left(s,x\right)}
	\\&\qquad\quad\,=
	\int_{\mathbb{R}^N}u_0\left(\left(I_{u,t}^{\ell^0}\right)^{\frac{1}{2}}\left(I_{u,t}^{\ell^1}\right)^{-\frac{1}{2}}\left(\xi'-e^{\left(t-u\right)A}y-F^{}_{u,t}\right)+e^{\left(t-u\right)A}y+F^{}_{u,t}\right)\\
	&\qquad\qquad\qquad\times \restr{\left\langle \left(I^{\ell^0}_{u,t}\right)^{-\frac{1}{2}}e^{\left(t-u\right)A}B\left(u,y\right),\left(I^{\ell^1}_{u,t}\right)^{-\frac{1}{2}}\left(\xi'-e^{\left(t-u\right)A}y-F^{}_{u,t}\right)\right\rangle\phi_{u,t}^{\ell^1}\left(y,\xi'\right)d\xi'\,}{y=Z^{\ell^1}_u\left(s,x\right)}.
\end{align*}
At this point we invoke the disintegration formula of the conditional expectation (see, e.g., \cite[Theorem $5.4$]{ola}) and \eqref{trans} to deduce \eqref{cond_k01}, completing the proof.
\end{proof}
\begin{rem}\label{cambia}
	The function $k_{u,t}^{\ell^0},\,\ell^0\in\mathbb{S},\,0\le u<t\le T,$ does not depend on the probability space where the underlying OU processes $Z^{\ell^0}_t\left(u,x\right),\,{x\in\mathbb{R}^N},$ are defined.
\end{rem}
	\subsection{Random time--shift}\label{sub1.2}
	Here we investigate the subordinated Brownian motion case (see Lemma \ref{l13}) after some further preparation. In what follows, we denote by $\Omega_k,\,k\in\mathbb{N}\cup \left\{0\right\},$ copies of the probability space $\Omega$. Let $\mathbb{W}$ be the space of continuous functions from $\mathbb{R}_+$ to $\mathbb{R}^N$ vanishing at $0$ and endow it with the Borel $\sigma$--algebra $\mathcal{B}\left(\mathbb{W}\right)$ associated with the topology of locally uniform convergence. The pushforward probability measure generated by  $W\left(\cdot\right)\colon\left(\Omega,\mathcal{F},\mathbb{P}\right)\to\left(\mathbb{W}, \mathcal{B}\left(\mathbb{W}\right) \right)$ is denoted by $\mathbb{P}^\mathbb{W}$ and makes the canonical process $\mathfrak{x}=\left(x_t\right)_{t\ge 0}$ a Brownian motion. 
	We work with the usual completion $\left(\mathbb{W}, \overline{\mathcal{B}\left(\mathbb{W}\right)}, \overline{\mathbb{P}^\mathbb{W}} \right)$ of this probability space:  $\mathfrak{x}$ is still a Brownian motion with respect to its minimal augmented filtration (cfr. \cite[Theorem $7.9$]{KS}). The completeness of the space $\left(\Omega,\mathcal{F},\mathbb{P}\right)$ implies the measurability of $W\left(\cdot\right)\colon\left(\Omega,\mathcal{F},\mathbb{P}\right)\to\left(\mathbb{W}, \overline{\mathcal{B}\left(\mathbb{W}\right)} \right)$ and the fact that $\overline{\mathbb{P}^\mathbb{W}}$ is still the pushforward probability measure generated by $W\left(\cdot\right)$. Since $W\left(\cdot\right)$ is independent from $\mathcal{F}^L$, a regular conditional distribution of $W\left(\cdot\right)$ given $\mathcal{F}^L$ is $\overline{\mathbb{P}^\mathbb{W}}\left(A\right),\,A\in\overline{\mathcal{B}\left(\mathbb{W}\right)}$. 	Moreover, we denote by (coherently with Subsection~\ref{sub1.1})
	\[
	\widebar{Z}^{\ell}_t\left(u,y\right)=e^{\left(t-u\right)A}y+F^{}_{u,t}+\int_{u}^{t}e^{\left(t-r\right)A}\sqrt{Q}\,dx_{\ell_r}\colon \mathbb{W}\to\mathbb{R}^N,\quad  0\le u\le t\le T,\,y\in\mathbb{R}^N,\,\ell\in\mathbb{S},
	\]
	and by $\mathbb{E}_k\left[\cdot\right]$ [resp., $\mathbb{E}_{\mathbb{W}}\left[\cdot\right]$] the expectation of a random variable defined on $\Omega_k$ [resp., $\mathbb{W}$].
	We are now in position to prove the next claim, which is the analogue of \cite[Corollary $2.2$]{LRF2}. 
	\begin{lemma}\label{l13}
		For every $x\in\mathbb{R}^N$ and $0\le s<t\le T$ one has
		\begin{align}
			\begin{split}\label{k^0}
				v^1_s\left(t,x\right)=\int_{s}^{t}du\left(\mathbb{E}_0\otimes\mathbb{E}_1\right)\Bigg[
				\\
				u_0\left(\left(I_{u,t}^{L}\left(\omega_0\right)\right)^{\frac{1}{2}}\left(\left(I_{u,t}^{L}
				\right)^{-\frac{1}{2}}\left(Z^{s,x}_{t}-e^{\left(t-u\right)A}Z^{s,x}_{u}-F_{u,t}\right)\right)\left(\omega_1\right)+F_{u,t}+e^{\left(t-u\right)A}Z^{s,x}_{u}\left(\omega_1\right)\right)\\
				\times\left\langle \left(I^{L}_{u,t}\left(\omega_0\right)\right)^{-\frac{1}{2}}e^{\left(t-u\right)A}B\left(u,Z_{u}^{s,x}\left(\omega_1\right)\right),\left(I^{L}_{u,t}\right)^{-\frac{1}{2}}
				\left(Z^{s,x}_{t}-e^{\left(t-u\right)A}Z^{s,x}_{u}-F_{u,t}\right)\left(\omega_1\right)\right\rangle\Bigg].
			\end{split}
		\end{align}
	\end{lemma}
\begin{proof}
	Fix $0\le s <t\le T$; combining the definition in \eqref{k^0_def} and the expression in \eqref{no_bel} we get, by the law of total expectation, for every $u\in\left(s,t\right),$
	\begin{equation}\label{m}
		k^0_{u,t}\left(y\right)
		=\mathbb{E}_0\left[\mathbb{E}_0\left[u_0\left(Z_{t}^{u,y}\right)\left\langle \left(I_{u,t}^L\right)^{-1}e^{\left(t-u\right)A}B\left(u,y\right),Z^{u,y}_{t}-e^{\left(t-u\right)A}y-F_{u,t}\right\rangle\Big|\mathcal{F}^L\right]\right],\quad y\in\mathbb{R}^N.
	\end{equation}
The discussion preceding this lemma together with the usual rules of change of probability space (see, e.g., \cite[§X-$2$]{JA}) and the substitution formula in  \cite[Lemma $5$]{AB} lets us apply the disintegration formula for the conditional expectation to get, from \eqref{det_1}-\eqref{m} and Remark \ref{cambia},
\begin{multline}\label{1}
	k^0_{u,t}\left(y\right)=\mathbb{E}_0\left[\restr{\mathbb{E}_{\mathbb{W}}\left[u_0\left(\widebar{Z}^{\ell^0}_{t}\left(u,y\right)\right)\left\langle \left(I_{u,t}^{\ell^0}\right)^{-1}e^{\left(t-u\right)A}B\left(u,y\right),\widebar{Z}^{\ell^0}_t\left(u,y\right)-e^{\left(t-u\right)A}y-F^{}_{u,t}\right\rangle\right]}{\ell^0=L\left(\omega_0\right)}\right]\\
	=\mathbb{E}_0\left[\restr{k^{\ell^0}_{u,t}\left(y\right)}{\ell^0=L\left(\omega_0\right)}\right],\quad y\in\mathbb{R}^N.
\end{multline}
Since we aim to compute \eqref{1goal}, for a generic $x\in\mathbb{R}^N$ we focus on
\begin{equation}\label{2}
	R_{s,u}k^0_{u,t}\left(x\right)
	=
	\mathbb{E}_1\left[\mathbb{E}_1\left[k^0_{u,t}\left(Z^{s,x}_{u}\right)\Big|\mathcal{F}^L\right]\right]
	=\mathbb{E}_1\left[\restr{\mathbb{E}_{\mathbb{W}}\left[k^0_{u,t}\left(\widebar{Z}^{\ell^1}_{u}\left(s,x\right)\right)\right]}{\ell^1=L\left(\omega_1\right)}\right],
\end{equation}
with the last equality which is obtained by the same argument as in \eqref{1}. At this point we combine \eqref{1} and \eqref{2} to write, using Fubini's theorem,
\begin{align*}
	R_{s,u}k^0_{u,t}\left(x\right)
	=
	\mathbb{E}_0\left[\restr{\mathbb{E}_1\left[
		\restr{\mathbb{E}_{\mathbb{W}}\left[k^{\ell^0}_{u,t}\left(\widebar{Z}^{\ell^1}_{u}\left(s,x\right)\right)\right]}{\ell^1=L\left(\omega_1\right)}\right]}{\ell^0=L\left(\omega_0\right)}\right].
\end{align*} 
Recalling that \eqref{cond_k01} in Lemma \ref{det_time_1} provides us with an expression for $k^{\ell^0}_{u,t}\left(\widebar{Z}^{\ell^1}_u\left(s,x\right)\right)$, we can use the law of total expectation and reason backwards with the conditioning in $\mathcal{F}^L$ to conclude that 
\begin{align*}
	R_{s,u}k^0_{u,t}\left(x\right)
	\\=\mathbb{E}_0\Bigg[\mathbb{E}_1
	\bigg[u_0\left(\left(I_{u,t}^{\ell^0}\right)^{\frac{1}{2}}\left(\left(I_{u,t}^{L}\right)^{-\frac{1}{2}}\left(Z^{s,x}_{t}-e^{\left(t-u\right)A}Z^{s,x}_{u}-F_{u,t}\right)\right)\left(\omega_1\right)+F^{}_{u,t}+e^{\left(t-u\right)A}Z^{s,x}_{u}\left(\omega_1\right)\right)\\
	\times\left\langle \left(I^{\ell^0}_{u,t}\right)^{-\frac{1}{2}}e^{\left(t-u\right)A}B\left(u,Z_{u}^{s,x}\left(\omega_1\right)\right),\left(I^{L}_{u,t}\right)^{-\frac{1}{2}}
	\left(Z^{s,x}_{t}-e^{\left(t-u\right)A}Z^{s,x}_{u}-F_{u,t}\right)\left(\omega_1\right)\right\rangle\bigg]\bigg|_{\ell^0=L\left(\omega_0\right)}\Bigg].
\end{align*}
Integrating the previous expression in the interval $\left(s,t\right)$ with respect to $u$ we obtain \eqref{k^0} completing the proof.
\end{proof}
	\section{The general term of the iteration scheme}\label{sec_3}
	Let $\alpha\in\left(\frac{1}{2},1\right).$ We want to analyze the general term $v^{n+1}_s\left(t,x\right)=\int_{s}^{t}R_{s,u}k^{n}_{u,t}\left(x\right)du,\,0\le s <t\le T,$ of the iteration \eqref{it_scheme} for an integer $n\ge1$. Therefore we search for an explicit expression of
	\begin{equation}\label{n_goal}
		R_{s,u}k^n_{u,t}\left(x\right)=\mathbb{E}\left[k^n_{u,t}\left(Z_{u}^{s,x}\right)\right],\quad x\in \mathbb{R}^N,\,u\in\left(s,t\right).
	\end{equation}
	\subsection{Deterministic time--shift}
	We continue the construction carried out in Subsection \ref{sub1.1}. Specifically,  fix an integer $n\ge 1$ and $t\in\left(0,T\right]$; for every $i=1,\dots,n$, $\left(i+1\right)-$tuple $\left(s_{n-i+1},s_{n-i+2},\dots ,s_{n+1}\right)$ such that $0\le s_{n-i+1}<s_{n-i+2}<\dots<s_{n+1}<t$ and  $\ell^0,\dots,\ell^{i}\in\mathbb{S}$  we define (see \eqref{k0,1})
	\begin{equation}\label{it_det_def}
		k^{\ell^0,\dots,\ell^i}_{s_{n-i+1},\dots,s_{n+1},t}\left(y\right)=
		\left\langle
		B\left(s_{n-i+1},y\right),\nabla^\top\mathbb{E}\left[k^{\ell^0,\dots,\ell^{i-1}}_{s_{n-i+2},\dots,s_{n+1},t}\left(Z_{s_{n-i+2}}^{\ell^{i}}\left(s_{n-i+1},y\right)\right)\right]
		\right\rangle,\quad y\in\mathbb{R}^N.
	\end{equation}
	Note that, by the continuity and boundedness of $B$, an induction argument shows that all these functions are well defined and in $C_b\left(\mathbb{R}^N\right).$ Moreover, as in Remark \ref{cambia} we observe that their value does not depend on the probability space where the underlying OU processes are constructed. By  \eqref{der_deter} we have ($y\in\mathbb{R}^N$)
	\begin{multline}\label{it_det}
		k^{\ell^0,\dots,\ell^i}_{s_{n-i+1},\dots,s_{n+1},t}\left(y\right)=\mathbb{E}\bigg[k^{\ell^0,\dots,\ell^{i-1}}_{s_{n-i+2},\dots, s_{n+1},t}\left(Z_{s_{n-i+2}}^{\ell^{i}}\left(s_{n-i+1},y\right)\right)\\
		\times \left\langle  \left(I_{s_{n-i+1},s_{n-i+2}}^{\ell^{i}}\right)^{-1}e^{\left(s_{n-i+2}-s_{n-i+1}\right)A}B\left(s_{n-i+1},y\right),Z^{\ell^{i}}_{s_{n-i+2}}\left(s_{n-i+1},0\right)-F^{}_{s_{n-i+1},s_{n-i+2}}\right\rangle\bigg].
	\end{multline}
	To shorten the notation, in what follows we set $n_i= n-i$.
	Once again, motivated by \eqref{n_goal} we want to find  an explicit formula for the term $k^{\ell^0,\dots,\ell^n}_{s_1,\dots,s_{n+1},t}\left(Z^{\ell^{n+1}}_{s_1}\left(s,x\right)\right)$, where $\ell^{n+1}\in\mathbb{S},\,0\le s<s_1<\dots<s_{n+1}<t$ and $x\in\mathbb{R}^N$.  A candidate for such an expression is given by \eqref{cond_k01} in Lemma \ref{det_time_1}, from which we deduce the next claim.
	\begin{lemma}\label{L8}
		Consider $0\le s <t\le T$ and an integer $n\ge 1$. Then, for every $x\in\mathbb{R}^N$, $i=0,\dots,n$,  $\left(i+2\right)-$tuple $\left(s_{n_i},s_{n_i+1},\dots,s_{n+1}\right)$ such that $s\le s_{n_i}< s_{n_{i}+1}<\dots<s_{n+1}<t$ and $\ell^0,\dots,\ell^{i+1}\in\mathbb{S}$, one has
	\begin{align}\label{conde_kn}
		 \notag k^{\ell^0,\dots,\ell^i}_{s_{n_i+1},\dots,s_{n+1},t}\left(Z_{s_{n_i+1}}^{\ell^{i+1}}\left(s_{n_i},x\right)\right)
		 =
		 \mathbb{E}\Bigg[
		 u_0\bigg(\sum_{j=1}^{i+1}e^{\left(t-s_{n_j+3}\right)A}\bigg[F^{}_{s_{n_j+2},s_{n_j+3}}+\left(I^{\ell^{j-1}}_{s_{n_j+2},s_{n_j+3}}\right)^{\frac{1}{2}}\left(I^{\ell^{i+1}}_{s_{n_j+2},s_{n_j+3}}\right)^{-\frac{1}{2}}\\
		 \notag\left(Z_{s_{n_j+3}}^{\ell^{i+1}}\left(s_{n_i},x\right)-e^{\left(s_{n_j+3}-s_{n_j+2}\right)A}Z_{s_{n_j+2}}^{\ell^{i+1}}\left(s_{n_i},x\right)-F^{}_{s_{n_j+2},s_{n_j+3}}\right)\bigg]
		 +e^{\left(t-s_{n_i+1}\right)A}Z_{s_{n_i+1}}^{\ell^{i+1}}\left(s_{n_i},x\right)\bigg)
		 	\\\notag\times\prod_{j=1}^{i+1}	
		 	\Big\langle \left(I^{\ell^{j-1}}_{s_{n_j+2},s_{n_j+3}}\right)^{-\frac{1}{2}}e^{\left(s_{n_j+3}-s_{n_j+2}\right)A}B\bigg(s_{n_j+2},
		 	\sum_{k=j+1}^{i+1}			{\color{black}{e^{\left(s_{n_j+2}-s_{n_k+3}\right)A}}}\Big[\left(I^{\ell^{k-1}}_{s_{n_k+2},s_{n_k+3}}\right)^{\frac{1}{2}}\left(I^{\ell^{i+1}}_{s_{n_k+2},s_{n_k+3}}\right)^{-\frac{1}{2}}\\\notag
		 	\left(Z_{s_{n_k+3}}^{\ell^{i+1}}\left(s_{n_i},x\right)-e^{\left(s_{n_k+3}-s_{n_k+2}\right)A}Z_{s_{n_k+2}}^{\ell^{i+1}}\left(s_{n_i},x\right)-F^{}_{s_{n_k+2},s_{n_k+3}}\right)+F^{}_{s_{n_k+2},s_{n_k+3}}\Big]
		 	+e^{\left(s_{n_j+2}-s_{n_i+1}\right)A}\\\notag Z_{s_{n_i+1}}^{\ell^{i+1}}\left(s_{n_i},x\right)\bigg),
		 	\left(I^{\ell^{i+1}}_{s_{n_j+2},s_{n_j+3}}\right)^{-\frac{1}{2}}
		 	\left(Z_{s_{n_j+3}}^{\ell^{i+1}}\left(s_{n_i},x\right)-e^{{\left(s_{n_j+3}-s_{n_j+2}\right)}A}Z_{s_{n_j+2}}^{\ell^{i+1}}\left(s_{n_i},x\right)-F^{}_{s_{n_j+2},s_{n_j+3}}\right)\Big\rangle 
		 	\\\bigg|\,\sigma\left(Z^{\ell^{i+1}}_{s_{n_i+1}}\left(s_{n_i},x\right)\right)\Bigg],\quad \mathbb{P}-\text{a.s.},
		\end{align}
	where $s_{n+2}=t$.
	\end{lemma}
	In the previous expression, we interpret the empty sum to be $0$: we adopt this convention hereafter.
	\begin{proof}
		Fix $0\le s < t\le T$ and  an integer $n\ge 1$. We proceed by induction on $i$, observing that the base case $i=0$ has been proven in \eqref{cond_k01}, where $s_n=s$ and $s_{n+1}=u$.
		
		For the induction step, suppose that the statement is valid for $i=m-1$, for some  $m=1,\dots,n$: our goal is to show that it holds true for $i=m$, as well. Take an $\left(m+2\right)-$tuple $\left(s_{n_m},s_{n_m+1},\dots,s_{n+1}\right)$ such that $s\le s_{n_m}<s_{n_m+1}<\dots<s_{n+1}<t$ and $\ell^0,\dots,\ell^{m+1}\in\mathbb{S}$; recalling \eqref{it_det} and denoting  $s_{n+2}=t$, we apply the inductive hypothesis and the law of total expectation to write, for every $y\in\mathbb{R}^N$,
		\begin{align}\label{omg}
			\notag k^{\ell^0,\dots,\ell^{m}}_{s_{n_m+1},\dots,s_{n+1},t}\left(y\right)
			=	
			\mathbb{E}\Bigg[
			u_0\bigg(\sum_{j=1}^{m}e^{\left(t-s_{n_j+3}\right)A}\bigg[F^{}_{s_{n_j+2},s_{n_j+3}}+\left(I^{\ell^{j-1}}_{s_{n_j+2},s_{n_j+3}}\right)^{\frac{1}{2}}\left(I^{\ell^{m}}_{s_{n_j+2},s_{n_j+3}}\right)^{-\frac{1}{2}}\\
			\notag
			\left(Z_{s_{n_j+3}}^{\ell^{m}}\left(s_{n_m+1},y\right)-e^{\left(s_{n_j+3}-s_{n_j+2}\right)A}Z_{s_{n_j+2}}^{\ell^{m}}\left(s_{n_m+1},y\right)-F^{}_{s_{n_j+2},s_{n_j+3}}\right)\bigg]
			+e^{\left(t-s_{n_m+2}\right)A}Z_{s_{n_m+2}}^{\ell^{m}}\left(s_{n_m+1},y\right)\bigg)\notag
			\\\notag\times\prod_{j=1}^{m}	
			\Big\langle \left(I^{\ell^{j-1}}_{s_{n_j+2},s_{n_j+3}}\right)^{-\frac{1}{2}}e^{\left(s_{n_j+3}-s_{n_j+2}\right)A}B\bigg(s_{n_j+2},
			\sum_{k=j+1}^{m}	e^{\left(s_{n_j+2}-s_{n_k+3}\right)A}\Big[\left(I^{\ell^{k-1}}_{s_{n_k+2},s_{n_k+3}}\right)^{\frac{1}{2}}\left(I^{\ell^{m}}_{s_{n_k+2},s_{n_k+3}}\right)^{-\frac{1}{2}}\\\notag
			\left(Z_{s_{n_k+3}}^{\ell^{m}}\left(s_{n_m+1},y\right)-e^{\left(s_{n_k+3}-s_{n_k+2}\right)A}Z_{s_{n_k+2}}^{\ell^{m}}\left(s_{n_m+1},y\right)-F^{}_{s_{n_k+2},s_{n_k+3}}\right)+F^{}_{s_{n_k+2},s_{n_k+3}}\Big]+e^{\left(s_{n_j+2}-s_{n_m+2}\right)A}\\\notag
			Z_{s_{n_m+2}}^{\ell^{m}}\!\left(s_{n_m+1},y\right)\!\!\bigg),\!
			\left(I^{\ell^{m}}_{s_{n_j+2},s_{n_j+3}}\right)^{-\frac{1}{2}}\!\!
			\left(Z_{s_{n_j+3}}^{\ell^{m}}\left(s_{n_m+1},y\right)-e^{{\left(s_{n_j+3}-s_{n_j+2}\right)}A}Z_{s_{n_j+2}}^{\ell^{m}}\left(s_{n_m+1},y\right)-F^{}_{s_{n_j+2},s_{n_j+3}}\right)\!\Big\rangle 
			\notag\\\times
			\left\langle  \left(I_{s_{n_m+1},s_{n_m+2}}^{\ell^{m}}\right)^{-1}e^{\left(s_{n_m+2}-s_{n_m+1}\right)A}B\left(s_{n_m+1},y\right),Z^{\ell^{m}}_{s_{n_m+2}}\left(s_{n_m+1},0\right)-F^{}_{s_{n_m+1},s_{n_m+2}}\right\rangle
			\Bigg],
		\end{align}
		where we also consider the $\sigma\left(Z^{\ell^{m}}_{s_{n_m+2}}\left(s_{n_m+1},y\right)\right)-$measurability of the random variable 
		\[
			\left\langle  \left(I_{s_{n_m+1},s_{n_m+2}}^{\ell^{m}}\right)^{-1}e^{\left(s_{n_m+2}-s_{n_m+1}\right)A}B\left(s_{n_m+1},y\right),Z^{\ell^{m}}_{s_{n_m+2}}\left(s_{n_m+1},0\right)-F^{}_{s_{n_m+1},s_{n_m+2}}\right\rangle.
		\]
		To shorten the notation we  write ($y\in\mathbb{R}^N$)
		\[
		k^{\ell^0,\dots,\ell^{m}}_{s_{n_m+1},\dots,s_{n+1},t}\left(y\right)=
		\mathbb{E}\left[f\left(Z^{\ell^{m}}_{s_{n_m+2}}\!\left(s_{n_m+1},y\right),Z^{\ell^{m}}_{s_{n_m+3}}\!\left(s_{n_m+1},y\right),\dots, Z^{\ell^{m}}_{s_{n+1}}\!\left(s_{n_m+1},y\right),Z^{\ell^{m}}_t\!\left(s_{n_m+1},y\right)\right)\right] .
		\] 
		Since $Z_r^{\ell^{m}}\left(s_{n_m+1},y\right),\,{r\in\left[s_{n_m+1},t\right]}$, is a Markov process,  we know that  (cfr.  \cite[Proposition $7.2$]{ola})
		\begin{multline*}
			\left(Z^{\ell^{m}}_{s_{n_m+2}}\left(s_{n_m+1},y\right),Z^{\ell^{m}}_{s_{n_m+3}}\left(s_{n_m+1},y\right),\dots, 
			Z^{\ell^{m}}_{s_{n+1}}\left(s_{n_m+1},y\right)
			,Z^{\ell^{m}}_t\left(s_{n_m+1},y\right)\right)
			\\\sim
			\mu^{\ell^{m}}_{s_{n_{m}+1},s_{n_m+2}}\left(y\right)\otimes \mu^{\ell^{m}}_{s_{n_{m}+2},s_{n_m+3}}\otimes \dots
			\otimes 
			\mu^{\ell^{m}}_{s_{n_{}}, s_{n+1}}
			\otimes\mu^{\ell^{m}}_{s_{n+1},t}.
		\end{multline*}
		Hence, using the same notation as in the previous section,
		\begin{multline}\label{dn_1}
			k^{\ell^0,\dots,\ell^{m}}_{s_{n_m+1},\dots,s_{n+1},t}\left(y\right)
			=\int_{\mathbb{R}^N}\phi^{\ell^{m}}_{s_{n_{m+1}},s_{n_m+2}}\left(y,\xi_1\right)\Bigg(\int_{\mathbb{R}^N}\phi^{\ell^{m}}_{s_{n_{m}+2},s_{n_m+3}}\left(\xi_1,\xi_2\right)\bigg(\dots\bigg(
			\\
			\int_{\mathbb{R}^N}\phi^{\ell^{m}}_{s_{n+1},t}\left(\xi_m,\xi_{m+1}\right)
			f
			\left(\xi_1,\dots,\xi_{m+1}\right)d\xi_{m+1}\bigg)
			\dots
			\bigg)d\xi_2\Bigg)d\xi_1.
		\end{multline}
		We wish to rewrite the expression in \eqref{dn_1} as an integral with respect to
		\begin{equation*}
			\mu^{\ell^{m+1}}_{s_{n_{m}+1},s_{n_m+2}}\left(y\right)\otimes \mu^{\ell^{m+1}}_{s_{n_{m}+2},s_{n_m+3}}\otimes \dots
			\otimes\mu^{\ell^{m+1}}_{s_{n},s_{n+1}}
			\otimes\mu^{\ell^{m+1}}_{s_{n+1},t}.
		\end{equation*}
		In order to do so, we sequentially perform the following substitutions:
		\begin{equation*}
			\begin{cases}
				\xi_1=\left(I_{s_{n_m+1},s_{n_m+2}}^{\ell^{m}}\right)^{\frac{1}{2}}
				\left(I_{s_{n_m+1},s_{n_m+2}}^{\ell^{m+1}}\right)^{-\frac{1}{2}}
				\left(\xi_1'-e^{\left(s_{n_m+2}-s_{n_m+1}\right)A}y-F^{}_{s_{n_m+1},s_{n_m+2}}\right)\\
				\qquad\qquad\qquad\qquad\qquad\qquad\qquad\qquad	+e^{\left(s_{n_m+2}-s_{n_m+1}\right)A}y+F^{}_{s_{n_m+1},s_{n_m+2}}\eqqcolon g_1\left(\xi_1'\right);
				\\
				\xi_{h}=\left(I_{s_{n_m+h},s_{n_m+h+1}}^{\ell^{m}}\right)^{\frac{1}{2}}
				\left(I_{s_{n_m+h},s_{n_m+h+1}}^{\ell^{m+1}}\right)^{-\frac{1}{2}}
				\left(\xi_h'-e^{\left(s_{n_m+h+1}-s_{n_m+h}\right)A}\xi'_{h-1}-F^{}_{s_{n_m+h},s_{n_m+h+1}}\right)\\
				\qquad\quad	+e^{\left(s_{n_m+h+1}-s_{n_m+h}\right)A}g_{h-1}\left(\xi'_1,\dots\xi'_{h-1}\right)+F^{}_{s_{n_m+h},s_{n_m+h+1}}\eqqcolon  g_{h}\left(\xi'_1,\dots,\xi'_{h}\right),\quad h=2,\dots,m+1.
			\end{cases}
		\end{equation*}
	In this way, \eqref{dn_1} becomes
		\begin{multline*}
			k^{\ell^0,\dots,\ell^{m}}_{s_{n_m+1},\dots,s_{n+1},t}\left(y\right)
			=
			\int_{\mathbb{R}^N}
			\phi^{\ell^{m+1}}_{s_{n_{m}+1},s_{n_m+2}}\left(y,\xi'_1\right)
			\Bigg(\int_{\mathbb{R}^N}\phi^{\ell^{m+1}}_{s_{n_{m}+2},s_{n_m+3}}\left(\xi'_1,\xi'_2\right)\bigg(\dots\bigg(
			\\
			\int_{\mathbb{R}^N}\phi^{\ell^{m+1}}_{s_{n+1},t}\left(\xi'_m,\xi'_{m+1}\right)
			f
			\left(g_1\left(\xi'_1\right),\dots,g_{m+1}\left(\xi'_1,\dots\xi'_{m+1}\right)\right)d\xi'_{m+1}\bigg)
			\dots
			\bigg)d\xi'_2\Bigg)d\xi'_1.
		\end{multline*}
		Expanding the notation for $f$ contained in \eqref{omg}, we can exploit several cancellations to get
		\begin{align}\label{si}
			\notag	k^{\ell^0,\dots,\ell^m}_{s_{n_m+1},\dots,s_{n+1},t}\left(y\right)
			=\int_{\mathbb{R}^N}	\phi^{\ell^{m+1}}_{s_{n_{m}+1},s_{n_m+2}}\left(y,\xi'_1\right)\Bigg(\int_{\mathbb{R}^N}
			\phi^{\ell^{m+1}}_{s_{n_{m}+2},s_{n_m+3}}\left(\xi'_1,\xi'_2\right)\bigg(\dots\bigg(	\int_{\mathbb{R}^N}\phi^{\ell^{m+1}}_{s_{n+1},t}\left(\xi_m',\xi_{m+1}'\right)\\\notag
			u_0\bigg(\sum_{j=1}^{m+1}e^{\left(t-s_{n_j+3}\right)A}\bigg[F^{}_{s_{n_j+2},s_{n_j+3}}+\left(I^{\ell^{j-1}}_{s_{n_j+2},s_{n_j+3}}\right)^{\frac{1}{2}}\left(I^{\ell^{m+1}}_{s_{n_j+2},s_{n_j+3}}\right)^{-\frac{1}{2}}\\
			\notag
			\left(\xi'_{m_j+2}-e^{\left(s_{n_j+3}-s_{n_j+2}\right)A}\xi'_{m_j+1}-F^{}_{s_{n_j+2},s_{n_j+3}}\right)\bigg]
			+e^{\left(t-s_{n_m+1}\right)A}y\bigg)\\
			\notag
			\times\prod_{j=1}^{m+1}	
			\Big\langle \left(I^{\ell^{j-1}}_{s_{n_j+2},s_{n_j+3}}\right)^{-\frac{1}{2}}e^{\left(s_{n_j+3}-s_{n_j+2}\right)A}B\bigg(s_{n_j+2},
			\sum_{k=j+1}^{m+1}{\color{black}{e^{\left(s_{n_j+2}-s_{n_k+3}\right)A}}}\Big[\left(I^{\ell^{k-1}}_{s_{n_k+2},s_{n_k+3}}\right)^{\frac{1}{2}}\left(I^{\ell^{m+1}}_{s_{n_k+2},s_{n_k+3}}\right)^{-\frac{1}{2}}\\\notag
			\left(\xi'_{m_k+2}-e^{\left(s_{n_k+3}-s_{n_k+2}\right)A}\xi'_{m_k+1}-F^{}_{s_{n_k+2},s_{n_k+3}}\right)+F^{}_{s_{n_k+2},s_{n_k+3}}\Big]
			+e^{\left(s_{n_j+2}-s_{n_m+1}\right)A}y\bigg),
			\\
			\left(I^{\ell^{m+1}}_{s_{n_j+2},s_{n_j+3}}\right)^{-\frac{1}{2}}
			\left(\xi'_{m_j+2}-e^{{\left(s_{n_j+3}-s_{n_j+2}\right)}A}\xi'_{m_j+1}-F^{}_{s_{n_j+2},s_{n_j+3}}\right)\Big\rangle 
			d\xi_{m+1}'\bigg)\dots\bigg)d\xi_{2}'\Bigg)d\xi_1',
		\end{align}
		where we denote by $\xi_0'=y$.
		Noticing that
		$
		\delta_{Z^{\ell^{m+1}}_{s_{n_m+1}}\left(s_{n_m},x\right)}\otimes 	\mu^{\ell^{m+1}}_{s_{n_{m}+1},s_{n_m+2}}\otimes \dots
		\otimes \mu^{\ell^{m+1}}_{s_{n},s_{n+1}}
		\otimes\mu^{\ell^{m+1}}_{s_{n+1},t},\,x\in\mathbb{R}^N,
		$ is a regular conditional distribution for
		\[
		\mathbb{P}\left(\left(Z^{\ell^{m+1}}_{s_{n_m+1}}\left(s_{{n_m}},x\right),Z^{\ell^{m+1}}_{s_{n_m+2}}\left(s_{{n_m}},x\right),\dots,Z^{\ell^{m+1}}_{s_{n+1}}\left(s_{{n_m}},x\right),Z^{\ell^{m+1}}_{t}\left(s_{{n_m}},x\right)\right)\in\cdot\,\Big|\,\sigma\left(Z^{\ell^{m+1}}_{s_{n_m+1}}\left(s_{{n_m}},x\right)\right)\right)
		\]
		thanks to  \cite[Propositions $5.6$-$7.2$]{ola}, 
		\eqref{si} yields \eqref{conde_kn} by the disintegration formula of the conditional expectation. The proof is now complete.
	\end{proof}
	\subsection{Random time--shift}
	We argue  by conditioning with respect to $\mathcal{F}^L$ as in Subsection \ref{sub1.2}. First, we present a result which generalizes \eqref{1} in the proof of Lemma \ref{l13}.  
	\begin{lemma}\label{8}
	Consider $0\!\le\!s <\!t \le\!T$\,and an integer $n\ge1$.\,Then for every $i=0,\dots,n,\, s_{1}\in\left(s,t\right)$ and~$y\in\mathbb{R}^N$, 
		\begin{equation}\label{relation_n}
			k^i_{s_1,t}\left(y\right)=\int_{s_1}^{t}ds_{2}\int_{s_2}^{t}ds_{3}\dots\int_{s_{i}}^{t}ds_{i+1}
			\mathbb{E}_{i}\left[\dots\left[\mathbb{E}_0\left[
			k^{\ell^0,\dots,\ell^i}_{s_1,\dots,s_{i+1},t}\left(y\right)
			\Big|_{\ell^0=L\left(\omega_0\right)}\right]\dots\right]\bigg|_{\ell^{i}=L\left(\omega_{i}\right)}
			\right].
		\end{equation}
	\end{lemma}
In this expression,  we ignore the time--integrals when $i=0$.
	\begin{proof}
		Take an integer $n\ge 1$ and proceed by induction on $i$. For $i=0$, there are no integrals in time in \eqref{relation_n}, which then reduces to \eqref{1} with $s_1=u$.
		
		Suppose that the statement holds true for $i=m-1$, for some  $m=1,\dots,n$: we want to prove its validity also for $i=m$. In order to do so, let us fix $y\in\mathbb{R}^N$ and $s_1\in\left(s,t\right)$; recalling the definition of $k^m_{s_1,t}$ in \eqref{it_scheme}, by  Lemma \ref{cont_k} we can apply    \eqref{no_bel} to get
		\begin{multline*}
			k^{m}_{s_{1},t}\left(y\right)
			=
			\int_{s_1}^{t}ds_2\, \mathbb{E}_{m}\Bigg[
			\mathbb{E}_{\mathbb{W}}\bigg[
			k^{m-1}_{s_2,t}\left(\widebar{Z}^{\ell^{m}}_{s_2}\left(s_1,y\right)\right)\\
			\times \left\langle \left(I_{s_1,s_2}^{\ell^{m}}\right)^{-1}e^{\left(s_{2}-s_1\right)A}B\left(s_1,y\right), \widebar{Z}^{\ell^{m}}_{s_2}\left(s_1,y\right)
			-e^{\left(s_{2}-s_1\right)A}y-F_{s_1,s_2}^{}
			\right\rangle\bigg]\bigg|_{\ell^{m}=L\left(\omega_{m}\right)}\Bigg].
		\end{multline*}
		By the inductive hypothesis, we substitute the expression for $k_{s_2,t}^{m-1},\,s_2\in\left(s_1,t\right),$ in the previous equality to obtain (ignoring the inner time--integral when $m=1$)
		\begin{align*}
			k^{m}_{s_{1},t}\left(y\right)
			=
			\int_{s_1}^{t}ds_2\,
			\mathbb{E}_{m}\Bigg[\mathbb{E}_{\mathbb{W}_{}}\Bigg[\\
			\int_{s_2}^{t}ds_{3}\dots\int_{s_m}^{t}ds_{m+1}\
			\mathbb{E}_{m-1}\Bigg[\dots\left[\mathbb{E}_0\left[
			k^{\ell^0,\dots,\ell^{m-1}}_{s_2,\dots,s_{m+1},t}\left(\widebar{Z}^{\ell^{m}}_{s_{2}}\left(s_1,y\right)\right)
			\bigg|_{\ell^0=L\left(\omega_0\right)}\right]\dots\right]\Bigg|_{\ell^{m-1}=L\left(\omega_{m-1}\right)}
			\Bigg]			
			\\
			\times\left\langle \left(I_{s_1,s_2}^{\ell^{m}}\right)^{-1}e^{\left(s_{2}-s_1\right)A}B\left(s_1,y\right), \widebar{Z}^{\ell^{m}}_{s_2}\left(s_1,y\right)
			-e^{\left(s_{2}-s_1\right)A}y-F^{}_{s_1,s_2}\right\rangle
			\Bigg]\Bigg|_{\ell^{m}=L\left(\omega_{m}\right)}\Bigg]
			,
		\end{align*}
		which we rewrite by Fubini's theorem --whose application is guaranteed by \cite[Lemma $2.12$]{LRF1}, upon carrying out computations similar to those in the proof of \cite[Theorem $6$]{AB} (see also \cite[Proposition $3.2$]{BB})-- as follows:
		\begin{align*}
			k^{m}_{s_{1},t}\left(y\right)
			=
			\int_{s_1}^{t}ds_2\int_{s_2}^{t}ds_3\dots\int_{s_m}^{t}ds_{m+1}
			\mathbb{E}_{m}\Bigg[		
			\mathbb{E}_{m-1}\bigg[\dots\bigg[\mathbb{E}_0\bigg[\mathbb{E}_{\mathbb{W}}\bigg[\\
				k^{\ell^0,\dots,\ell^{m-1}}_{s_2,\dots,s_{m+1},t}\left(\widebar{Z}^{\ell^{m}}_{s_{2}}\left(s_1,y\right)\right)
				\left\langle \left(I_{s_1,s_2}^{\ell^{m}}\right)^{-1}e^{\left(s_{2}-s_1\right)A}B\left(s_1,y\right), \widebar{Z}^{\ell^{m}}_{s_2}\left(s_1,y\right)
				-e^{\left(s_{2}-s_1\right)A}y-F^{}_{s_1,s_2}\right\rangle\\
				\bigg]\bigg|_{\ell^0=L\left(\omega_0\right)}\bigg]\dots\bigg]\bigg|_{\ell^{m-1}=L\left(\omega_{m-1}\right)}	
				\bigg]\bigg|_{\ell^{m}=L\left(\omega_{m}\right)}\Bigg]
				.
			\end{align*}
			This provides us with \eqref{relation_n}, once we plug in the expression of $k^{\ell^0,\dots,\ell^m}_{s_1,\dots,s_{m+1},t}\left({y}\right)$ in \eqref{it_det}. Thus, the proof is complete.
		\end{proof}
		According to \eqref{n_goal}, given $0\le s <s_1<t\le T$ we are interested in
		\begin{multline}\label{3_1}
			R_{s,s_1}k^n_{s_1,t}\left(x\right)=\mathbb{E}_{n+1}\left[\restr{\mathbb{E}_{\mathbb{W}_{}}\left[k^n_{s_1,t}\left(\widebar{Z}^{\ell^{n+1}}_{s_1}\left(s,x\right)\right)\right]}{\ell^{n+1}=L\left(\omega_{n+1}\right)}\right]
			=
			\int_{s_1}^{t}ds_{2}\int_{s_2}^{t}ds_{3}\dots\int_{s_{n}}^{t}ds_{n+1}\\
			\mathbb{E}_{0}\left[\dots\left[\mathbb{E}_{n+1}\left[\mathbb{E}_{\mathbb{W}_{}}
			\left[k^{\ell^0,\dots,\ell^{n}}_{s_1,\dots,s_{n+1},t}\left(\widebar{Z}^{\ell^{n+1}}_{s_1}\left(s,x\right)\right)\right]
			\Big|_{\ell^{n+1}=L\left(\omega_{n+1}\right)}\right]\dots\right]\bigg|_{\ell^{0}=L\left(\omega_{0}\right)}
			\right],
		\end{multline}
where we use Lemma \ref{8} and Fubini's theorem for the second equality.	Since Lemma \ref{L8} in the previous subsection gives us a formula for $k^{\ell^0,\dots,\ell^n}_{s_1,\dots,s_{n+1},t}\left(\widebar{Z}^{\ell^{n+1}}_{s_1}\left(s,x\right)\right)$ (see \eqref{conde_kn} with $s_0=s$ and $i=n$), we just plug it into \eqref{3_1}, apply the law of total expectation and reason backwards with the conditioning in $\mathcal{F}^L$ to deduce the next result (cfr. \cite[Theorem $2.3$]{LRF2}).
		\begin{theorem}\label{general}
			For every $x\in\mathbb{R}^N$ and $0\le s<t\le T$ one has
			\begin{align}\label{gene_ite}
				\notag		v^{n+1}_s\left(t,x\right)=\int_{s}^{t}ds_1\int_{s_1}^{t}ds_2\dots\int_{s_n}^{t}ds_{n+1}\left(\mathbb{E}_0\otimes \mathbb{E}_1\otimes\dots\otimes  \mathbb{E}_{n+1}\right)\Bigg[
				\\\notag
				u_0\bigg(\sum_{j=1}^{n+1}e^{\left(t-s_{n_j+3}\right)A}\bigg[F_{s_{n_j+2},s_{n_j+3}}+\left(I^{L}_{s_{n_j+2},s_{n_j+3}}\left(\omega_{j-1}\right)\right)^{\frac{1}{2}}\\\left(\left(I^{L}_{s_{n_j+2},s_{n_j+3}}\right)^{-\frac{1}{2}}
				\notag\left(Z_{s_{n_j+3}}^{s,x}-e^{\left(s_{n_j+3}-s_{n_j+2}\right)A}Z_{s_{n_j+2}}^{s,x}-F_{s_{n_j+2},s_{n_j+3}}\right)\right)\left(\omega_{n+1}\right)\bigg]
				+e^{\left(t-s_{1}\right)A}Z_{s_1}^{s,x}\left(\omega_{n+1}\right)\bigg)
				\\\notag\times\prod_{j=1}^{n+1}	
				\Big\langle \left(I^{L}_{s_{n_j+2},s_{n_j+3}}\left(\omega_{j-1}\right)\right)^{-\frac{1}{2}}e^{\left(s_{n_j+3}-s_{n_j+2}\right)A}B\bigg(s_{n_j+2},
				\sum_{k=j+1}^{n+1}			{\color{black}{e^{\left(s_{n_j+2}-s_{n_k+3}\right)A}}}\Big[\left(I^{L}_{s_{n_k+2},s_{n_k+3}}\left(\omega_{k-1}\right)\right)^{\frac{1}{2}}
				\\\notag
				\left(\left(I^{L}_{s_{n_k+2},s_{n_k+3}}\right)^{-\frac{1}{2}}
				\left(Z_{s_{n_k+3}}^{s,x}-e^{\left(s_{n_k+3}-s_{n_k+2}\right)A}Z_{s_{n_k+2}}^{s,x}-F_{s_{n_k+2},s_{n_k+3}}\right)\right)\left(\omega_{n+1}\right)
				+F_{s_{n_k+2},s_{n_k+3}}\Big]
				\\+e^{\left(s_{n_j+2}-s_{1}\right)A}
				Z_{s_1}^{s,x}\left(\omega_{n+1}\right)\bigg),
				\left(I^{L}_{s_{n_j+2},s_{n_j+3}}\right)^{-\frac{1}{2}}\!
				\left(Z_{s_{n_j+3}}^{s,x}-e^{{\left(s_{n_j+3}-s_{n_j+2}\right)}A}Z_{s_{n_j+2}}^{s,x}-F_{s_{n_j+2},s_{n_j+3}}\right)\left(\omega_{n+1}\right)\Big\rangle 
				\Bigg],
			\end{align}
			where $s_{n+2}=t.$
		\end{theorem}

			\section{Numerical simulations}\label{simulations}
			In this section we report on the results obtained by implementing the iterative scheme described above for two choices of the nonlinear vector field $B_0$. We interpret the SDE in \eqref{st_semi} as a finite--dimensional approximation of the reaction--diffusion SPDE
			\begin{equation*}
				\begin{cases}
					dX\left(t,\xi\right)=\left(\Delta X\left(t,\xi\right)+B_0\left(t,X\left(t,\xi\right)\right)\right)dt+\sigma \,dW_{L_t},&t\ge s,\\
					{X}\left(s,\xi\right)=x\left(\xi\right),&\xi\in\mathbb{T}^1,
				\end{cases}
			\end{equation*}
			where  $\mathbb{T}^1=\mathbb{R}^1/\mathbb{Z}^1$ is the one--dimensional torus (we refer to \cite[Example 1]{AB} for an accurate description of this framework). Hence we consider $\lambda_k=\left|k\right|^2,\,k=1,\dots,N$, and we take $Q=\sigma^2 \text{Id}.$ Here $\sigma>0$ is a parameter describing the strength of the noise. Before moving to the application of the model, we have to determine the \emph{time--shift} function $f\in C\left(\left[0,T\right];\mathbb{R}^N\right)$ appearing in the OU process $Z^{s,x},\,x\in\mathbb{R}^N$ (see \eqref{OU}). Since we are dealing with a rotation--invariant noise and $\alpha\in\left(\frac12,1\right)$, $\mathbb{E}\left[W_{L_t}\right]=0,\,t\ge0$. As a consequence, the choice of $f$ can be motivated as in \cite[Introduction]{LRF2} for the Brownian case. In brief, we consider 
			\[
				f\left(t\right)= B_0\left(t,x\left(t\right)\right),\quad t\in\left[0,T\right],
			\]
			where $x\left(\cdot\right)\colon \left[0,T\right]\to\mathbb{R}^N$ is the unique solution of the integral equation
			\begin{equation}\label{shift}
				x\left(t\right)=x+\int_{s}^{t}\left(Ax\left(r\right)+B_0\left(r,x\left(r\right)\right)\right)dr,\quad t\in\left[s,T\right],
			\end{equation}
		and $x\left(t\right)=x,\,t\in\left[0,s\right].$ 	Of course,  $x\left(\cdot\right)$ is computed numerically.  Note that \eqref{shift} is the deterministic counterpart of the semilinear SDE \eqref{st_semi}, and that the expected value function of the OU process coincides with $x\left(\cdot\right)$ in the interval $\left[s,T\right]$ by the choice of $f$. The intuition is that, at least when the noise is weak, the trajectories of the semilinear solutions are ``close'' to $x\left(\cdot\right),$ allowing the $0-$th iterate to perform better than it would do with $f\equiv0$. Figure \ref{fig2} clearly displays this idea in the case of (bounded) cubic nonlinearity treated below (see \eqref{bcn}). Furthermore, in the sequel we monitor the effect of the time--shift on the first order approximation provided by our scheme. All the simulations are carried out using the High Performance Computing Center of the Scuola Normale Superiore (\url{https://hpccenter.sns.it}).
			\begin{figure}[b]
			\centering
			{\includegraphics[height=.21\textheight,width=.49\textwidth]{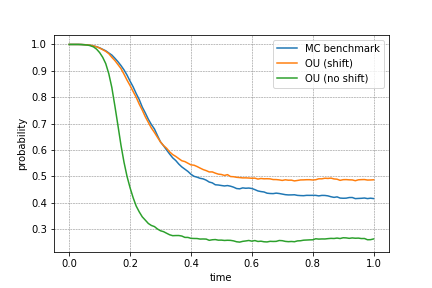}}
			\hfill
			{\includegraphics[height=.21\textheight, width=.49\textwidth]{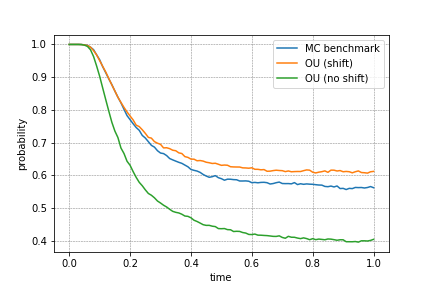}}
			\caption{Behavior in time of the OU approximations in the bounded cubic case with and without time--shift. The panel on the left refers to $\alpha=0.55$, the one on the right to $\alpha=0.85$. $\sigma=0.5$ everywhere.}
			\label{fig2}
		\end{figure}\\
			We work in dimension $N=100$, with $u_0\left(x\right)=1_{\left\{\left|x\right|>R\right\}},\,x\in\mathbb{R}^N,$ for some $R>0$, and we denote by $\mathbf{e}\in\mathbb{R}^N$ the vector with all components equal to $1$. In particular, given $0\le s<t\le 1$, we are interested in applying our iterates to approximate $P_{s,t}u_0\left(\mathbf{e}\right)=\mathbb{P}\left(\left|X_t^{s,\mathbf{e}}\right|>R\right)$, whose reference value is computed by averaging $10^5$ samples of $X_t^{s,\mathbf{e}}$ obtained by the Euler--Maruyama scheme with time step $10^{-4}$. The same strategy is used to obtain the $0-$th iterate $v^0_s\left(t,\mathbf{e}\right)=\mathbb{P}\left(\left|Z^{s,\mathbf{e}}_t\right|>R\right)$. In order to calculate the numerical integrals appearing in the formulas for $v^n_s\left(t,\mathbf{e}\right),\,n\in\mathbb{N}$ (see \eqref{k^0}-\eqref{gene_ite}), we use left Riemann sums in a uniform grid with mesh $10^{-2}.$  We will keep track of the relative error  $\epsilon^n_r$, defined by
			\[			\epsilon^n_r=\frac{P_{s,t}u_0\left(\mathbf{e}\right)-\sum_{i=0}^{n}v^i_s\left(t,\mathbf{e}\right)}{P_{s,t}u_0\left(\mathbf{e}\right)},\quad n\in\mathbb{N}\cup\left\{0\right\}.
			\]
			Finally, we will mainly focus on the first iteration, with the aim of understanding the possible improvements that it provides over the linear approximation of the OU process. In fact, although it is possible to implement our scheme up to any order thanks to \eqref{gene_ite}, one needs an $n-$dimensional integral (in time) to get the iterate $v^n_s\left(t,\mathbf{e}\right),\,n\in\mathbb{N},$ fact which complicates the application of our method and may result in losing its computational advantage over the classical Euler--Maruyama approach.
			In what follows, we fix the initial time $s=0$ and  the threshold $R=1$. For the subordinator $L$, we set $\bar{\gamma}=1$ in \eqref{subo}.
			
			We first take $B_0\left(x\right)_k=\sin\left(x_k\right),\,k=1,\dots,N$.  Table \ref{t1} shows the performance of the first order approximation of the iterative scheme with time--shift as $\alpha$ varies in $\left(\frac{1}{2},1\right)$,  $\sigma=1$ and $t=1$. Table \ref{t2} is analogous, but it refers to $f\equiv0$ (no time--shift). The first thing we notice is that in both cases the first iteration improves on  the outcomes of the linear approximation. The role of the time--shift $f$ is evident in the column $\epsilon^0_r$: it allows $v^0_0\left(1,\mathbf{e}\right)$ to be closer to the benchmark probability, and the first iterate builds on this to guarantee a better overall performance, particularly when $\alpha$ is close to $\frac{1}{2}$.
			\begin{table}
				\centering
				\begin{tabular}{ | c| c| c| c| c|c |}
					\hline
					$\alpha$ & $\mathbb{P}(|X_1^{0,\mathbf{e}}|>1)$& $v^0_0\left(1,\mathbf{e}\right)$&$\epsilon^0_r$
					& $v^1_0\left(1,\mathbf{e}\right)$&$\epsilon^1_r$
					\\ \hhline{|=|=|=|=|=|=|}
					0.55 & 0.687&0.639&6.99e-2&0.012&5.24e-2 \\ \hline
					0.65 & 0.713&0.676&5.19e-2&1.34e-2&3.31e-2 \\ \hline
					0.75 & 0.794&0.737&7.18e-2&3.34e-2&2.97e-2 \\ \hline
					0.85 & 0.899&0.863&0.040&1.87e-2&1.92e-2 \\ 
					\hline
				\end{tabular}
				\caption{First order approximation in the sine case with time--shift; noise strength $\sigma=1$. }
				\label{t1}
				\vspace{1em}
				\centering
				\begin{tabular}{ | c| c| c| c| c|c |}
					\hline
					$\alpha$ & $\mathbb{P}(|X_1^{0,\mathbf{e}}|>1)$& $v^0_0\left(1,\mathbf{e}\right)$&$\epsilon^0_r$
					& $v^1_0\left(1,\mathbf{e}\right)$&$\epsilon^1_r$
					\\ \hhline{|=|=|=|=|=|=|}
					0.55 & 0.691&0.502&0.274&0.101& 0.127\\ \hline
					0.65 & 0.720&0.558&0.225&0.110& 7.22e-2\\\hline
					0.75 & 0.785&0.666&0.151&8.84e-2&0.039 \\\hline
					0.85&0.896&0.840&6.25e-2&3.86e-2&1.94e-2\\
					\hline
				\end{tabular}
				\caption{Same setting as in Table \ref{t1}, without time--shift ($f\equiv 0$).}
				\label{t2}
			\end{table}\\
			Next, Figure \ref{fig1} displays the behavior in time --up to $t=1$-- of the first order approximation in the case of time--shift for two strengths of noise ($\sigma=0.1$ and $\sigma=1.3$). Here $\alpha=0.6$ is fixed. The panels of this figure highlight the benefits of considering $v^1_0\left(\cdot,\mathbf{e}\right)$ over the starting OU estimates, especially when the noise is~weak. 
			
			Secondly, we analyze the polynomial vector field 
			\begin{equation}\label{bcn}
			B_0\left(x\right)_k=b_0\norm{\bar{y}}_\infty\frac{\left(\bar{y}_k-x_k\right)\left|\bar{y}_k-x_k\right|^2}{b_0\norm{\bar{y}}_\infty+{\mathcal{S}\left(\mathcal{S}^+\left(\bar{y}-x\right)\right)}^3},\quad k=1,\dots, N,
			\end{equation}
			where $\bar{y}\in\mathbb{R}^N,\,b_0>0,\,\mathcal{S}\colon \mathbb{R}^N\to \mathbb{R}$ and $\mathcal{S}^+\colon \mathbb{R}^N\to \mathbb{R}^N$, with ($x\in\mathbb{R}^N$)
			\[
				\mathcal{S}\left(x\right)=\frac{\sum_{i=1}^N x_ie^{ax_i}}{\sum_{i=1}^N e^{ax_i}}; 
				\qquad
				 \mathcal{S}^+\left(x\right)_k=\frac{x_ke^{ax_k}-x_ke^{-ax_k}}{e^{ax_k}+e^{-ax_k}}
				 ,\,k=1,\dots,N.
			\]
			The maps $\mathcal{S},\mathcal{S}^+$ are smooth approximations of the maximum function and replace the infinity norm in \eqref{bcn}, allowing $B_0\in C_b^{3}\left(\mathbb{R}^N;\mathbb{R}^N\right),$ coherently with our theoretical framework. Therefore $B_0$ is to be interpreted as a  cubic nonlinearity with a cutoff for large values of  $\norm{x}_\infty.$ For our experiments, we consider $b_0=2,\,\bar{y}=2\mathbf{e}$ and $a=1\text{e}\,04$.  In Tables \ref{t3}-\ref{t4} we report the outcomes of simulations with and without $f$, respectively, when $\sigma=0.7,\, t=1$ and $\alpha$ varies in $\left(\frac{1}{2},1\right).$  In particular, Table \ref{t3} shows that, in the case of time--shift, the first iterate always remarkably outperforms the linear approximation. On the contrary, when $f\equiv0$ (Table \ref{t4}), $v^1_0\left(1,\mathbf{e}\right)$ deteriorates the OU estimate, and we are forced to implement the second iterate to get an accuracy similar to the one provided by the time--shift (compare the columns $\epsilon^1_r$, Table \ref{t3}, and $\epsilon^2_r$, Table \ref{t4}). Of course, the trade--off in the introduction of $v^2_0\left(1,\mathbf{e}\right)$ consists in substantially increasing the computational time. \\
			Finally, in Figure \ref{fig3} we investigate the trajectories of $P_{0,\cdot}u_0\left(\mathbf{e}\right)$ and of the first order approximation in the time interval $\left[0,1\right]$, as well as the corresponding absolute relative errors. Here we fix $\alpha=0.6$ and consider two strengths of noise: $\sigma=0.1$ and $\sigma=1.3$. As already observed in the sine case, the advantages in introducing the first iterate are rather evident. Overall, we conclude that $v^1_0\left(\cdot, \mathbf{e}\right)$ proves to be a versatile and computationally cheap method to improve on the performances of the linear approximation.
						\vspace{1em}\begin{table}[h]
				\centering
				\begin{tabular}{ | c| c| c| c| c| c|}
					\hline
					$\alpha$ & $\mathbb{P}(|X_1^{0,\mathbf{e}}|>1)$& $v^0_0\left(1,\mathbf{e}\right)$&$\epsilon^0_r$
					& $v^1_0\left(1,\mathbf{e}\right)$&$\epsilon^1_r$
					\\ \hhline{|=|=|=|=|=|=|}
					0.55 & 0.501&0.562&-0.122&-5.19e-2&-1.82e-2\\ \hline
					0.65 & 0.531&0.594&-0.119&-6.65e-2&6.59e-3 \\ \hline
					0.75 & 0.587&0.648&-0.104& -6.40e-2&5.11e-3 \\ \hline
					0.85 & 0.679&0.743&-9.43e-2&-7.95e-2&2.28e-2 \\ \hline
				\end{tabular}
				\caption{First order approximation in the bounded cubic case with time--shift; noise strength $\sigma=0.7$. }
				\label{t3}
				\vspace{1em}
				\centering
				\begin{tabular}{ | c| c| c| c| c|c |c|c|}
					\hline
					$\alpha$ & $\mathbb{P}(|X_1^{0,\mathbf{e}}|>1)$& $v^0_0\left(1,\mathbf{e}\right)$&$\epsilon^0_r$
					& $v^1_0\left(1,\mathbf{e}\right)$&$\epsilon^1_r$&$v^2_0\left(1,\mathbf{e}\right)$&$\epsilon^2_r$
					\\ \hhline{|=|=|=|=|=|=|=|=|}
					0.55 & 0.495&0.374&0.244&-9.64e-4&0.246&0.109&2.62e-2\\ \hline
					0.65 & 0.536&0.396&0.261&-2.74e-2&0.312&0.142&4.74e-2 \\ \hline
					0.75 & 0.586&0.462&0.212&-8.16e-2&0.351&0.191&2.49e-2\\\hline
					0.85& 0.680&0.608&0.106&-8.20e-2&0.226&0.138&2.35e-2\\\hline
				\end{tabular}
				\caption{Same setting as in Table \ref{t3}, without time--shift ($f\equiv 0$).}
				\label{t4}
			\end{table}

\begin{figure}
	\centering
	{\includegraphics[height=0.21\textheight, width=.49\textwidth]{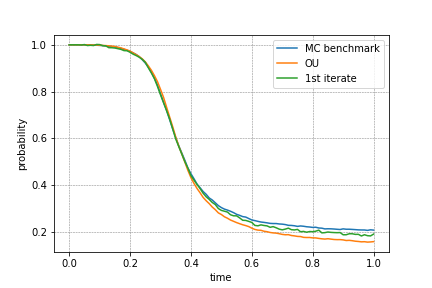}}
	\hfill
	{\includegraphics[height=0.21\textheight,width=.49\textwidth]{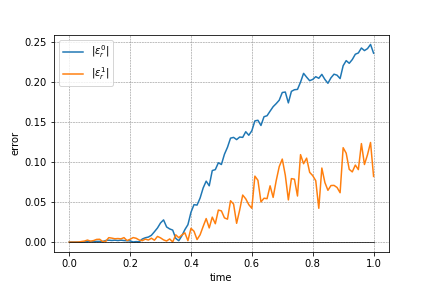}}\\
	{\includegraphics[height=0.21\textheight,width=.49\textwidth]{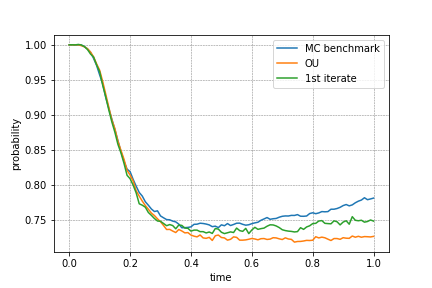}}
	\hfill
	{\includegraphics[height=0.21\textheight,width=.49\textwidth]{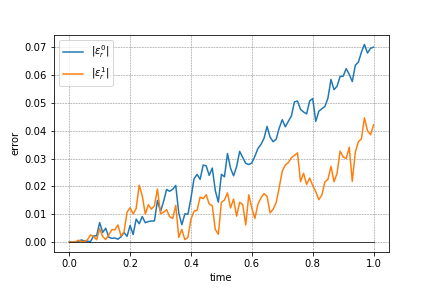}}
	\caption{Behavior in time of the first order approximation in the sine case with time--shift. In each line, the panel on the left shows the evolution of the probabilities, and the one on the right the corresponding errors. The top line refers to $\sigma=0.1$, the bottom line to $\sigma=1.3$. $\alpha=0.6$ everywhere.}
	\label{fig1}
\end{figure}
\begin{figure}
	\centering
	{\includegraphics[height=0.21\textheight,width=.49\textwidth]{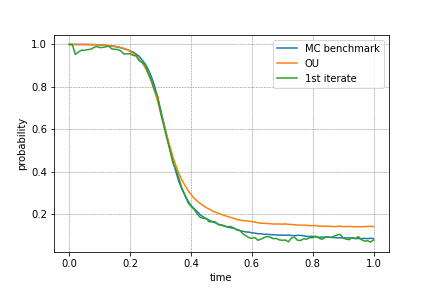}}
	\hfill
	{\includegraphics[height=0.21\textheight,width=.49\textwidth]{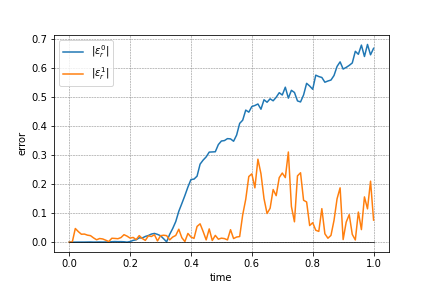}}\\
	{\includegraphics[height=0.21\textheight,width=.49\textwidth]{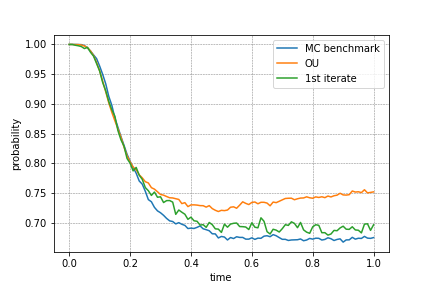}}
	\hfill
	{\includegraphics[height=0.21\textheight,width=.49\textwidth]{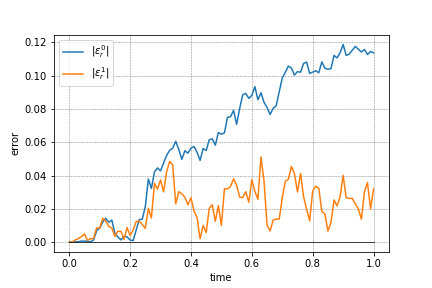}}
	\caption{Behavior in time of the first order approximation in the bounded cubic case with time--shift. In each line, the panel on the left shows the evolution of the probabilities, and the one on the right the corresponding errors. The top line refers to $\sigma=0.1$, the bottom line to $\sigma=1.3$. $\alpha=0.6$ everywhere.}
	\label{fig3}
\end{figure}
			
			\appendix
			\section{ Proof of Lemma \ref{hoc}}\label{appendix}
			In this appendix we provide the proof of Lemma \ref{hoc}, a useful result for the arguments of Section \ref{a_conn}.
			\begin{myproof}	{Lemma}{\ref{hoc}}
				Let us fix $0\le s \le T,\,x\in\mathbb{R}^N$ and a direction $h\in\mathbb{R}^N$; note that all the assertions of the statement are true for $\omega\in\Omega\setminus\Omega'$ by construction of the stochastic flow, hence we only focus on $\omega\in\Omega'$.  For every $\epsilon\in\left(0,1\right]$ and $t\in\left[s,T\right]$ define the incremental ratio function 
				\begin{align}\label{eq_diff}
					\notag Y_{x,h}^1\left(\epsilon,t\right)&=\epsilon^{-1}\left(X_t^{s,x+\epsilon h}\left(\omega\right)-X_t^{s,x}\left(\omega\right)\right)
					\\\notag &=
					h+\int_{s}^{t}\left(AY_{x,h}^1\left(\epsilon,r\right)+\frac{B_0\left(r,X_r^{s,x}\left(\omega\right)+\epsilon Y_{x,h}^1\left(\epsilon,r\right)\right)-B_0\left(r,X_r^{s,x}\left(\omega\right)\right)}{\epsilon}\right)dr
					\\&=h+\int_{s}^{t}\left(A+\int_{0}^{1}DB_0\left(r,X^{s,x}_r\left(\omega\right)+\rho\epsilon Y_{x,h}^1\left(\epsilon,r\right)\right) d\rho\right)Y_{x,h}^1\left(\epsilon,r\right)dr.
				\end{align}
				Notice that, for every $\epsilon\in\left(0,1\right]$ (omitting $\omega$ to keep notation short)
				\[
				\left|X^{s,x+\epsilon h}_t-X^{s,x}_t\right|\le \epsilon \left|h\right|+\left(\left|A\right|+\norm{DB_0}_{T,\infty}\right)\int_{s}^{t}\left|X^{s,x+\epsilon h}_r-X^{s,x}_r\right| dr,\quad t\in\left[s,T\right],
				\]
				where we recall that $\norm{DB_{0}}_{T,\infty}=\sup_{ 0\le t \le T}\norm{DB_0\left(t,\cdot\right)}_\infty$. 
				Thus, an application of Gronwall's lemma shows that $\left|Y_{x,h}^1\left(\epsilon,t\right)\right|\le \left|h\right|e^{\left(\left|A\right|+\norm{DB_0}_{T,\infty}\right)T}\eqqcolon C_1$ for all $t\in\left[s,T\right]$ and $\epsilon\in\left(0,1\right]$.
				Next, taking $\epsilon_1,\,\epsilon_2\in\left(0,1\right]$ and  $t\in\left[s,T\right]$ we compute from \eqref{eq_diff} 
				\begin{align}\label{come_prima}
					\notag	\left|Y_{x,h}^1\left(\epsilon_2,t\right)-Y_{x,h}^1\left(\epsilon_1,t\right)\right|\le 
					\int_{s}^{t}\left|A\right|\left|Y_{x,h}^1\left(\epsilon_2,r\right)-Y_{x,h}^1\left(\epsilon_1,r\right)\right|dr \\
					+\left|\int_{s}^{t}\!\bigg(\!\!\int_{0}^{1}DB_0\left(r,X^{s,x}_r\!+\rho\epsilon_2 Y_{x,h}^1\left(\epsilon_2,r\right)\right) d\rho\,Y_{x,h}^1\!\left(\epsilon_2,r\right)\notag
					-\!\int_{0}^{1}\!\!DB_0\left(r,X^{s,x}_r+\rho\epsilon_1 Y_{x,h}^1\left(\epsilon_1,r\right)\right) d\rho\, Y_{x,h}^1\left(\epsilon_1,r\right)\!\!\bigg)dr\right|\\\notag
					\le\!
					\left(\left|A\right|\!+\!\norm{DB_0}_{T,\infty}\right) \!\!\int_{s}^{t}\left|Y_{x,h}^1\left(\epsilon_2,r\right)-Y_{x,h}^1\left(\epsilon_1,r\right)\right|dr+\frac{N^{{2}}}{2}C_1\norm{\partial^2B_0}_{T,\infty}\!
					\int_{s}^{t}\!\left|\epsilon_2Y_{x,h}^1\left(\epsilon_2,r\right)-\epsilon_1Y_{x,h}^1\left(\epsilon_1,r\right)\right|\!dr\\
					\le \left(\left|A\right|+\norm{DB_0}_{T,\infty}+\frac{N^{2}}{2}C_1\norm{\partial^2B_0}_{T,\infty}\right)\int_{s}^{t}\left|Y_{x,h}^1\left(\epsilon_2,r\right)-Y_{x,h}^1\left(\epsilon_1,r\right)\right|dr+\frac{{N^{2}}}{2}C_1^2T\norm{\partial^2B_0}_{T,\infty}\left|\epsilon_2-\epsilon_1\right|,
				\end{align}
				where $\norm{\partial^2B_0}_{T,\infty}=\sup_{ 0\le t \le T}\norm{\partial^2B_0\left(t,\cdot\right)}_\infty$.
				Therefore another application of Gronwall's lemma shows that  the mapping $\epsilon\mapsto Y_{x,h}^1\left(\epsilon,t\right)$ is Lip--continuous in $\left(0,1\right]$ uniformly in  $t\in\left[s,T\right]$, and by the theorem of extension of uniformly continuous functions we obtain the existence of $D_hX^{s,x}_t\left(\omega\right)$. Now by dominated convergence we are allowed to pass to the limit in \eqref{eq_diff}, which yields
				\begin{equation}\label{der_h}
					D_hX^{s,x}_t\left(\omega\right)=h+\int_{s}^{t}\left(A+DB_0\left(r,X_r^{s,x}\left(\omega\right)\right)\right)D_hX^{s,x}_r\left(\omega\right)dr,\quad t\in\left[s,T\right].
				\end{equation}
				Given the arbitrarity of $h,\,x\in\mathbb{R}^N$, this equation shows that the mapping $x\mapsto X^{s,x}_t\left(\omega\right)$ belongs to $C^1\left(\mathbb{R}^N\right),$ with $\norm{DX^{s,\cdot}_t\left(\omega\right)}_\infty\le N\exp\left\{\left(\left|A\right|+\norm{DB_0}_{T,\infty}\right)T\right\}$.\\ In order to analyze higher--order derivatives, we work by induction; fix $m=1,\dots, n-1$ and suppose as inductive hypothesis that $X^{s,\cdot}_t\left(\omega\right)\in C^m\left(\mathbb{R}^N\right),\,t\in\left[s,T\right]$, with the estimate in \eqref{boundala} holding true for a sum from $i=1$ to $i=m$. Moreover, assume that for every  multi--index $\mathbf{h}\in\left(\mathbb{N}\cup \left\{0\right\}\right)^N$ with length $1\le \norm{\mathbf{h}}_1\le m$ one has, for any $t\in\left[s,T\right]$ (omitting $\omega$)
				\begin{equation}\label{component}
					D_{\mathbf{h}}X^{s,x}_{t}=\delta_\mathbf{h}+\int_{s}^t\left(\left(A+DB_{0}\left(r,X^{s,x}_{r}\right)\right) D_{\mathbf{h}}X^{s,x}_{r} +\mathcal{L}_{\mathbf{h}}\left(r,x\right)\right)dr,
					\quad 
					\delta_\mathbf{h}=\begin{cases}
						e_j,& \text{if }\norm{\mathbf{h}}_1=1 \text{ and } h_j=1,\\
						0,&\text{elsewhere}.
					\end{cases}
				\end{equation}
				Here $\left(e_j\right)_{j=1,\dots,N}$ is the canonical basis of $\mathbb{R}^N$ and  $\mathcal{L}_{\mathbf{h}}\left(t,x\right)=\left(\mathcal{L}_{\mathbf{h},j}\left(t,x\right)\right)_{j=1,\dots, N},$ with $\mathcal{L}_{\mathbf{h},j}\left(t,x\right)\in\mathbb{R}$ denoting a sum of products where one factor is a (partial) derivative at $X_t^{s,x}$ of $B_{0,j}\left(t,\cdot\right)$ up to order $\norm{\mathbf{h}}_1$ and the others are (partial) derivatives at $x$ of $X^{s,\cdot}_t$ up to order $\norm{\mathbf{h}}_1-1$. In particular, $\mathcal{L}_{\mathbf{h}}\left(t,x\right)=0$ when $\norm{\mathbf{h}}_1=1$ (cfr. \eqref{der_h}). At this point, consider  $x,\,h\in\mathbb{R}^N$ and fix a multi--index $\mathbf{h}$ with length $\norm{\mathbf{h}}_1=m$; by analogy with \eqref{eq_diff}, for any $\epsilon\in\left(0,1\right]$ and $t\in\left[s,T\right]$ define the incremental ratio function 
				\begin{multline*}
					Y^{m+1}_{x,h}\left(\epsilon,t\right)=
					\epsilon^{-1}\left(D_{\mathbf{h}}X_{t}^{s,x+\epsilon h}-D_{\mathbf{h}}X_{t}^{s,x}\right)
					\\=\int_{s}^{t}
					\bigg(\left(A+ D B_{0}\left(r,X^{s,x}_{r}\right)\right) Y^{m+1}_{x,h}\left(\epsilon,r\right)+\frac{D B_{0}\left(r,X^{s,x+\epsilon h }_{r}\right)-D B_{0}\left(r,X^{s,x}_{r}\right)}{\epsilon}D_{\mathbf{h}}X^{s,x+\epsilon h}_{r}
					\\+\epsilon^{-1}\left({\mathcal{L}_{\mathbf{h}}\left(r,x+\epsilon h\right)-\mathcal{L}_{\mathbf{h}}\left(r,x\right)}\right)\!\bigg)dr.
				\end{multline*}
				Note that for any $j=1,\dots, N$ we can write ($t\in\left[s,T\right]$,  $\epsilon\in\left(0,1\right]$)
				\[
				\epsilon^{-1} \left({D B_{0}\left(t,X^{s,x+\epsilon h }_{t}\right)-D B_{0}\left(t,X^{s,x}_{t}\right)}\right)_{j,\cdot}
				=
				\left(	\left(\int_{0}^{1}D^2B_{0,j}\left(t, X^{s,x}_t+\rho\epsilon Y_{x,h}^1\left(\epsilon,t\right)\right)d\rho\right) Y^1_{x,h}\left(\epsilon,t\right)\right)^\top, 
				\]
				and that, further, the inductive hypothesis of boundedness for the derivatives of $X^{s,\cdot}_t$ (see \eqref{boundala}), together with the structure of $\mathcal{L}_\mathbf{h}$ and $B_0\in C^{m+1}_b\left(\left[0,T\right]\times \mathbb{R}^N; \mathbb{R}^N\right)$ ensures that 
				\[
				\epsilon^{-1}\left|\mathcal{L}_{\mathbf{h}}\left(t,x+\epsilon h\right)-\mathcal{L}_{\mathbf{h}}\left(t,x\right)\right|\le C_2 \left|h\right|,\quad t\in\left[s,T\right], \, \epsilon\in\left(0,1\right],
				\]
				for some constant $C_2=C_2\left(A,B_0,T,m, N\right)>0$. These facts, the Lip--continuity of the map $\epsilon\mapsto Y_{x,h}^1\left(\epsilon,t\right)$ in $\left(0,1\right]$ uniformly in $t\in\left[s,T\right]$ and computations analogous to those in \eqref{come_prima} entail that there exists $D_h D_{\mathbf{h}}X^{s,x}_t\left(\omega\right)$. The arbitrarity of $x,\,h$ and $\mathbf{h}$ coupled with Gronwall's lemma provides us with the desired bound \eqref{boundala} for the derivatives of order $m+1$, and finally by dominated convergence 
				the validity of \eqref{component} for a multi--index of length $m+1$ is a consequence of the chain rule. In particular, $X^{s,\cdot}_t\left(\omega\right)\in C^{m+1}\left(\mathbb{R}^N\right)$. The proof is then complete, considering that the base case is provided by \eqref{der_h}.
			\end{myproof}

		\end{document}